\documentclass[11pt]{amsart}
\usepackage{amssymb,amsfonts,amsmath}
\usepackage[dvips]{graphicx}

\usepackage{psfrag}

\newdimen\margin   
\def\COMMENT#1{}
\def\TASK#1{}

\newtheorem{theo}{Theorem} 
\newtheorem{prop}[theo]{Proposition}
\newtheorem{lemma}[theo]{Lemma}

\newcommand{\mc}[1]{\mathcal{#1}}
\newcommand{\mb}[1]{\mathbb{#1}}
\newcommand{\nib}[1]{\noindent {\bf #1}}
\newcommand{\lra}{\leftrightarrow}
\newcommand{\sm}{\setminus}
\newcommand{\ov}{\overline}
\newcommand{\eps}{\varepsilon}
\newcommand{\ns}[1]{\hspace{-2pt} #1 \hspace{-2pt}}

\def\C{\mathcal{C}}

\def\eul{\rm{e}}

\begin{document}
\date{} 
\title{An exact minimum degree condition for Hamilton cycles in oriented graphs}
\author{Peter Keevash, Daniela K\"{u}hn \and Deryk Osthus}
\thanks {P.~Keevash was partially supported by the EPSRC, grant no.~EP/E02162X/1
and the NSF, grant DMS-0555755.
D.~K\"uhn was partially supported by the EPSRC, grant no.~EP/F008406/1.
D.~Osthus was partially supported by the EPSRC, grant no.~EP/E02162X/1 and~EP/F008406/1.}
\begin{abstract}
We show that every sufficiently large oriented graph~$G$ with
$\delta^+(G),\delta^-(G) \ge  \frac{3n-4}{8}$ contains a Hamilton cycle.
This is best possible and solves a problem of Thomassen from 1979. 
\end{abstract}
\maketitle

\section{Introduction}

A central topic in graph theory is that of giving conditions under which a graph
is Hamiltonian. One such result is the classical theorem of Dirac~\cite{D}, which
states that any graph on $n\ge 3$ vertices with minimum degree
at least $n/2$ contains a Hamilton cycle. For an analogue in directed graphs it is
natural to consider the \emph{minimum semi-degree~$\delta^0(G)$} of a digraph $G$,
which is the minimum of its minimum outdegree~$\delta^+(G)$ and its minimum
indegree~$\delta^-(G)$. The corresponding result is a theorem of Ghouila-Houri
\cite{GhouilaHouri}, which states that any digraph on $n$ vertices with minimum semi-degree
at least $n/2$ contains a Hamilton cycle. (When referring to paths and cycles in directed
graphs we always mean that these are directed, without mentioning this explicitly.)
Both of these results are best possible.

In 1979 Thomassen \cite{thomassen_79_long_cycles_constraints}
raised the natural corresponding question of determining the minimum
semi-degree that forces a Hamilton cycle in
an \emph{oriented graph} (i.e.~in a directed graph that can be obtained from a (simple)
undirected graph by orienting its edges).
Over the years since the question was posed, a series of improving bounds were
obtained in \cite{thomassen_81_long_cycles, thomassen_82, HaggkvistHamilton,
HaggkvistThomasonHamilton}. Recently, Kelly, K\"uhn and Osthus~\cite{KKO} were able to obtain
an approximate solution. They proved that
an oriented graph on $n$ vertices with minimum semi-degree at least
$(3/8+o(1))n$ has a Hamilton cycle, which asymptotically matches a lower bound given by
H\"aggkvist \cite{HaggkvistHamilton}. In this paper we obtain the following result
which exactly matches the lower bound of H\"aggkvist, thus answering Thomassen's question
for large oriented graphs.
\begin{theo} \label{main}
There exists a number $n_0$ so that any oriented graph
$G$  on $n \ge n_0$ vertices with minimum
semi-degree $\delta^0(G) \ge \left\lceil \frac{3n-4}{8} \right\rceil$ contains a Hamilton cycle.
\end{theo}
Note that Theorem~\ref{main} implies that every sufficiently large regular
tournament on~$n$ vertices contains at least $n/8$ edge-disjoint Hamilton cycles.
(To verify this, note that in a regular tournament, all in- and outdegrees are equal to $(n-1)/2$.
We can then greedily remove Hamilton cycles
as long as the degrees satisfy the condition in Theorem~\ref{main}.)%
\COMMENT{$(n-1)/2- \lceil (3n-4)/8  \rceil +1 \ge n/8$ }
This is a slight improvement on the $(1/8+o(1))n$ bound obtained in~\cite{KKO} by the same argument.
It is the best bound so far towards the classical conjecture of Kelly~(see e.g.~\cite{digraphsbook}), which states that
every regular tournament on~$n$ vertices can be
partitioned into~$(n-1)/2$ edge-disjoint Hamilton cycles.

H\"aggkvist~\cite{HaggkvistHamilton} also made the following conjecture which is closely related
to Theorem~\ref{main}. Given an oriented graph~$G$, let~$\delta(G)$ denote the minimum degree of~$G$
(i.e.~the minimum number of edges incident to a vertex) and set
$\delta^*(G):=\delta(G)+\delta^+(G)+\delta^-(G)$.
H\"aggkvist conjectured that if $\delta^*(G)>(3n-3)/2$, then $G$ has a Hamilton cycle.
(Note that this conjecture does not quite imply Theorem~\ref{main} as it results in a marginally
greater minimum semi-degree condition.)
In~\cite{KKO}, this conjecture was verified approximately, i.e.~if
$\delta^*(G) \ge (3/2+o(1))n$, then $G$ has a Hamilton cycle.
It seems possible that our approach can be extended to obtain an exact solution to this problem, but
this would certainly require some additional ideas beyond those applied here.

Our argument can be extended to find a cycle of any length~$\ell$, with $\ell \ge n/10^{10}$ (say)
through any given vertex.
We indicate the necessary modifications to the proof
in the final paragraph, the details can be found in~\cite{lukethesis}.
This result is used in~\cite{KKO2} to obtain the following pancyclicity result: any sufficiently large oriented
graph~$G$ with $\delta^0(G)\ge (3n-4)/8$ contains a cycle of length $\ell$ for all $\ell = 3,\dots ,n$.
For $\ell =4,\dots,n$, it even contains a cycle of length $\ell$ through a given vertex.

The rest of this paper is organized as follows. The next section contains some
basic notation. In Section~3 we describe the extremal example showing that
Theorem \ref{main} is best possible. We set up our main tools in Section~4, these
being a digraph form of the Regularity Lemma due to Alon and Shapira \cite{AlonShapiraTestingDigraphs}
and the Blow-up Lemma, both in the original form of Koml\'os, S\'ark\"ozy and
Szemer\'edi~\cite{KSSblowup} and in a recent stronger and more technical
form due to Csaba~\cite{csaba_06_bollobas_eldridge}. Our argument uses the
stability method, and so falls naturally into two cases, according to whether or not
our given oriented graph $G$ is structurally similar to the extremal example described
in Section~3. In Section~5 we prove a lemma that enables us to find
a Hamilton cycle when $G$ is not structurally similar to the extremal example.
The argument in this case is based on that of~\cite{KKO}.
Then we prove our main theorem in the final section.

\section{Notation}

Given two vertices~$x$ and~$y$ of
a directed graph~$G$, we write~$xy$ for the edge directed from~$x$ to~$y$.
The order~$|G|$ of~$G$ is the number of its vertices.
We write~$N^+_G(x)$ for the outneighbourhood of a vertex~$x$ and $d^+_G(x):=|N^+_G(x)|$ for its outdegree.
Similarly, we write~$N^-_G(x)$ for the inneighbourhood of~$x$ and $d^-_G(x):=|N^-_G(x)|$ for its indegree.
We write $N_G(x):=N^+_G(x)\cup N^-_G(x)$ for the neighbourhood of~$x$ and $d_G(x):=|N_G(x)|$ for its degree.
We use~$N^+(x)$ etc.~whenever
this is unambiguous. We write~$\Delta(G)$ for the maximum of $|N(x)|$ over all vertices $x\in G$.
Given a set~$A$ of vertices of~$G$, we write $N^+_G(A)$ for the set of all outneighbours of vertices in~$A$.
So~$N^+_G(A)$ is the union of $N^+_G(a)$ over all $a\in A$. $N^-_G(A)$ is defined similarly.
The directed subgraph of~$G$ induced by~$A$ is denoted by~$G[A]$ and we write~$E(A)$ for
the set of its edges and put $e(A):=|E(A)|$. If $S$ is a subset of the vertex set of $G$ then
$G-S$ denotes the digraph obtained from $G$ by deleting $S$ and all edges incident to $S$.

Recall that when referring to paths and cycles in directed graphs we always mean that
they are directed without mentioning this explicitly. Given two vertices~$x$ and~$y$ on a
directed cycle~$C$, we write $xCy$ for the subpath of~$C$ from $x$ to~$y$.
Similarly, given two vertices~$x$ and~$y$ on a directed path~$P$
such that~$x$ precedes~$y$, we write $xPy$ for the subpath of~$P$ from $x$ to~$y$.
A \emph{walk} in a directed graph~$G$ is a sequence of (not necessarily distinct)
vertices $v_1,v_2,\ldots,v_{\ell}$ where $v_iv_{i+1}$ is an edge for all $1\leq i<\ell$.
The walk is \emph{closed} if $v_1=v_\ell$. A \emph{$1$-factor} of~$G$ is a collection
of disjoint cycles which cover all the vertices of~$G$.
We define things similarly for undirected graphs.

Given two vertices $x,y$ of~$G$, an \emph{$x$-$y$ path} is a directed path which joins~$x$ to~$y$.
Given two disjoint subsets~$A$ and~$B$ of vertices of~$G$, an~$A$-$B$ edge is an edge~$ab$
where $a\in A$ and $b\in B$. We write $E(A,B)$ for the set of all these edges and
put $e(A,B):=|E(A,B)|$. We sometimes also write $e_G(A,B)$ for $e(A,B)$.
We denote by $(A,B)_G$ the oriented bipartite subgraph of~$G$
whose vertex classes are~$A$ and~$B$ and whose edge set is~$E(A,B)$.

We call an orientation of a complete graph a \emph{tournament} and an orientation of a
complete bipartite graph a \emph{bipartite tournament}.
Throughout the paper we omit floors and ceilings whenever this does not affect the
argument. For a positive integer $k$ we write $[k]:=\{1,\dots,k\}$.

\section{The extremal example}\label{extremal}

In this section, we show that the bound in Theorem~\ref{main} is best possible for every $n$.
The construction is due to H\"aggkvist~\cite{HaggkvistHamilton}. However, he only gives a
description for the case when $n$ is of the form $8k+7$. Below, we also add the remaining cases
to give the complete picture.

\begin{prop}\label{extremal_example}
For any $n\ge 3$ there is an oriented graph~$G$ on $n$ vertices
with minimum semi-degree~$\lceil(3n-4)/8\rceil -1$
which does not contain a $1$-factor, and so does not contain a Hamilton cycle.
\end{prop}
\begin{proof}
We construct $G$ as follows.
It has $n$ vertices partitioned into $4$ parts $A,B,C,D$, with $|B|>|D|$.
Each of $A$ and $C$ spans a tournament, $B$ and $D$ are joined by a bipartite tournament,
and we add all edges from $A$ to $B$, from $B$ to $C$, from $C$ to $D$ and from $D$ to $A$
(see Figure~1).
Since every path which joins two vertices in~$B$ has to pass through~$D$, it follows that
every cycle contains at least as many vertices from~$D$ as it contains from~$B$.
As $|B|>|D|$ this means that one cannot cover all the vertices of~$G$ by disjoint cycles,
i.e.~$G$ does not contain a 1-factor.
\begin{figure}\label{fig:extremal}
\centering\footnotesize
\psfrag{1}[][]{$A$}
\psfrag{2}[][]{$B$}
\psfrag{3}[][]{$C$}
\psfrag{4}[][]{$D$}
\includegraphics[scale=0.7]{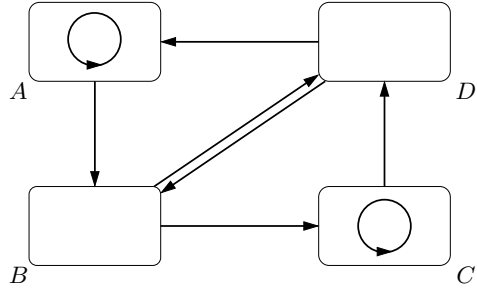}
\caption{The oriented graph in the proof of Proposition~\ref{extremal_example}.}
\end{figure}

It remains to show that the sizes of $A,B,C,D$ and the tournaments can be chosen to give

minimum semi-degree~$\delta^0(G)=\lceil(3n-4)/8\rceil -1$. The following table gives possible
values, according to the value of $n$ mod $8$:
$$
\begin{array}{l|l|l|l|l|l}
n & \lceil(3n-4)/8\rceil -1 & |A| & |B| & |C| & |D| \\
\hline
8k-1 & 3k-1 & 2k-1 & 2k+1 & 2k-1 & 2k \\
8k   & 3k-1 & 2k   & 2k+1 & 2k-1 & 2k \\
8k+1 & 3k-1 & 2k   & 2k+1 & 2k   & 2k \\
8k+2 & 3k   & 2k   & 2k+2 & 2k-1 & 2k+1 \\
8k+3 & 3k   & 2k   & 2k+2 & 2k   & 2k+1 \\
8k+4 & 3k   & 2k+1 & 2k+2 & 2k   & 2k+1 \\
8k+5 & 3k+1 & 2k+1 & 2k+2 & 2k+1 & 2k+1 \\
8k+6 & 3k+1 & 2k+2 & 2k+2 & 2k+1 & 2k+1
\end{array}
$$
We choose the tournaments inside $A$ and $C$ to be as regular as possible
(meaning that the indegree and outdegree of every vertex differ by at most $1$):
this can be achieved (e.g.) by arranging the vertices evenly around a circle
and directing edges so that the larger part of the circle lies to the right of each edge,
breaking ties arbitrarily.
For the bipartite tournament between $B$ and $D$ there are two cases. Firstly, when
$|B|=2k+1$ and $|D|=2k$ we choose it so that $|N^+(b) \cap D| = |N^-(b) \cap D| = k$
for every $b \in B$ and $\{|N^+(d) \cap B|, |N^-(d) \cap B|\} = \{k,k+1\}$ (in either
order) for every $d \in D$: this can be achieved by the blow-up of a $4$-cycle,
i.e., write $B=B_1 \cup B_2$, $D = D_1 \cup D_2$ with $|B_1|=k+1$, $|B_2|=|D_1|=|D_2|=k$
and direct edges from $B_1$ to $D_1$, from $D_1$ to $B_2$, from $B_2$ to $D_2$ and
from $D_2$ to $B_1$. Secondly, when $|B|=2k+2$ and $|D|=2k+1$ we choose the bipartite
tournament between~$B$ and~$D$ so
that $|N^+(b) \cap D|=k+1$ and $|N^-(b) \cap D|=k$ for every $b \in B$ and
so that for every $d\in D$ either $|N^+(d) \cap B|=|N^-(d) \cap B|=k+1$
or $|N^+(d) \cap B|=k$ and $|N^-(d) \cap B|=k+2$. This can
be achieved by labelling $B$ as $[2k+2]$, $D$ as $\mb{Z}/(2k+1)\mb{Z}$
and setting $N^+(b) \cap D := \{(k+1)b, (k+1)b+1, \dots, (k+1)b+k\}$ for $b \in B$.
Indeed, to check that the vertices in~$D$ have the correct indegrees note that
$N^+(b)\cap D$ is a segment of consecutive vertices of~$D$ for each $b\in B$
and the segment $N^+(b+1)\cap D$ starts immediately after $N^+(b)\cap D$
and ends with the first vertex of $N^+(b)\cap D$.

To verify the minimum semi-degree condition
we introduce the notation $\delta^+_S, \delta^-_S$ for the minimum in- and
outdegrees of vertices in $S \in \{A,B,C,D\}$. We can eliminate some of the
checking by noting that
we have always chosen $|C| \le |A|$ and $|D| < |B|$, so we have inequalities
$\delta_C^+ < \delta_C^- \le \delta_A^+$ and $\delta_C^+ \le \delta_A^-$ as well as
$\delta_B^+ \le \delta_D^-$. In the cases $n=8k-1,8k,8k+1$ we also have
$\delta_B^+ \le \delta_D^+$ and $\delta_B^+ \le \delta_B^-$, so
$$
\delta^0(G) = \min \{\delta_B^+, \delta_C^+\} =
\min \{|C|+\lfloor |D|/2 \rfloor, \lfloor (|C|-1)/2 \rfloor + |D|\} = 3k-1.
$$
In the case $n=8k+2$ we have $\delta_B^+ = |C| + (k+1) = 3k$,
$\delta_B^- = |A| + k = 3k$, $\delta_C^+ = \lfloor (|C|-1)/2 \rfloor + |D| = 3k$ and
$\delta_D^+ = |A| + k = 3k$, so $\delta^0(G)=3k$. The constructions for the cases $n=8k+3,8k+4$ are
obtained from $n=8k+2$ by increasing $|A|$ and/or $|C|$ by $1$, and we still
have $\delta^0(G) = \delta_C^+ = 3k$. When $n=8k+5$ we have
$\delta_B^+ = |C| + (k+1) = 3k+2$, $\delta_B^- = |A| + k = 3k+1$,
$\delta_C^+ = \lfloor (|C|-1)/2 \rfloor + |D| = 3k+1$ and
$\delta_D^+ = |A| + k = 3k+1$, so $\delta^0(G) = 3k+1$. Finally, the construction for the
case $n=8k+6$
is obtained from $n=8k+5$ by increasing $|A|$ by $1$ and we still have
$\delta^0(G)=\delta_C^+=3k+1$. In all cases we have $\delta^0(G) = \lceil(3n-4)/8\rceil -1$,
as required. \end{proof}

\nib{Remark.} One may add any number of edges that either go from $A$ to $C$ or lie within $D$
without creating a $1$-factor (although this does not increase the minimum semi-degree).

\section{The Diregularity Lemma, the Blow-up Lemma and other tools}\label{sec:reglem}
In this section we collect all the information we need about the
Diregularity Lemma and the Blow-up Lemma. See~\cite{KSi} for a survey on the Regularity Lemma
and~\cite{JKblowup} for a survey on the Blow-up Lemma.
We start with some more notation. The density of a bipartite graph $G = (A,B)$
with vertex classes~$A$ and~$B$ is defined to be
\[d_G(A,B) := \frac{e_G(A,B)}{|A||B|}.\]
We often write $d(A,B)$ if this is unambiguous. Given $\eps>0$, we say
that~$G$ is \emph{$\eps$-regular} if for all subsets $X\subseteq A$ and $Y\subseteq B$
with $|X|>\eps|A|$ and $|Y|>\eps|B|$ we have that
$|d(X,Y)-d(A,B)|<\eps$. Given $d\in [0,1]$ we say that~$G$ is $(\eps,d)$-\emph{super-regular}
if it is $\eps$-regular and furthermore $d_G(a)\ge (d-\eps) |B|$ for all $a\in A$ and
$d_G(b)\ge (d-\eps)|A|$ for all $b\in B$. (This is a slight variation of the standard definition
of $(\eps,d)$-super-regularity where one requires $d_G(a)\ge d |B|$ and
$d_G(b)\ge d|A|$.)
Given partitions $V_0,V_1,\ldots,V_k$ and $U_0,U_1,\dots,U_{\ell}$ of the vertex set of some graph, we say that
$V_0,V_1,\ldots,V_k$ \emph{refines} $U_0,U_1,\dots,U_{\ell}$ if for all $V_i$ with
$i\ge 1$ there is some $U_j$ with $j\ge 0$ which contains $V_i$. Note that $V_0$ need not be contained
in any $U_j$, so this is weaker than the usual notion of refinement of partitions.

The Diregularity Lemma is a version of the Regularity Lemma for digraphs due to
Alon and Shapira~\cite{AlonShapiraTestingDigraphs}. Its proof is quite similar to
the undirected version.
We will use the degree form of the Diregularity Lemma which can be easily derived
(see e.g.~\cite{young_05_extremal} for an explicit proof of all but the `refinement property')
from the standard version, in exactly the same manner
as the undirected degree form.

\begin{lemma}[Degree form of the Diregularity Lemma]\label{direg}
For every $\eps\in (0,1)$ and all numbers~$M'$, $M''$ there are numbers~$M$ and~$n_0$
such that if
\begin{itemize}
\item $G$ is a digraph on $n\ge n_0$ vertices,
\item $U_0,\dots,U_{M''}$ is a partition of the vertices of $G$,
\item $d\in[0,1]$ is any real number,
\end{itemize}
then there is a partition of the vertices of~$G$ into
$V_0,V_1,\ldots,V_k$ and a spanning subdigraph~$G'$ of~$G$ such that the following holds:
\begin{itemize}
\item $M'\le k\leq M$,
\item $|V_0|\leq \eps n$,
\item $|V_1|=\dots=|V_k|=:m$,
\item $V_0,V_1,\ldots,V_k$ refines the partition $U_0,U_1,\dots,U_{M''}$,
\item $d^+_{G'}(x)>d^+_G(x)-(d+\eps)n$ for all vertices $x\in G$,
\item $d^-_{G'}(x)>d^-_G(x)-(d+\eps)n$ for all vertices $x\in G$,
\item for all $i=1,\dots,k$ the digraph $G'[V_i]$ is empty,	
\item for all $1\leq i,j\leq k$ with $i\neq j$ the bipartite graph whose vertex classes
are~$V_i$ and~$V_j$ and whose edges are all the $V_i$-$V_j$ edges in~$G'$
is $\eps$-regular and has density either~$0$ or density at least~$d$.
\end{itemize}
\end{lemma}

$V_1,\ldots,V_k$ are called \emph{clusters}, $V_0$ is called the \emph{exceptional set}
and the vertices in~$V_0$ are called \emph{exceptional vertices}.
The last condition of the lemma says that all pairs of clusters are $\eps$-regular in both
directions (but possibly with different densities).
We call the spanning digraph~$G'\subseteq G$ given by the Diregularity Lemma the \emph{pure digraph}.
Given clusters $V_1,\ldots,V_k$ and a digraph~$G'$, the \emph{reduced digraph~$R'$
with parameters $(\eps,d)$} is the digraph
whose vertex set is~$[k]$ and in which~$ij$ is an edge
if and only if the bipartite graph whose vertex classes
are~$V_i$ and~$V_j$ and whose edges are all the $V_i$-$V_j$ edges in~$G'$
is $\eps$-regular and has density at least~$d$.%
\COMMENT{Since we are contiually modifying our partitions later on in Case 1, this seems to be
more suitable than the definition in the approximate paper}
(So if $G'$ is the pure digraph, then $ij$ is an edge in $R'$ if and only if there is a
$V_i$-$V_j$ edge in~$G'$.)

It is easy to see that the reduced digraph $R'$ obtained from the regularity lemma
`inherits' the minimum degree of~$G$, in that $\delta^+(R')/|R'| > \delta^+(G)/|G| - d - 2\eps$
and $\delta^-(R')/|R'| > \delta^-(G)/|G| - d - 2\eps$.
However,~$R'$ is not necessarily oriented even if the original digraph~$G$ is.
The next lemma (which is essentially from~\cite{KKO})
shows that by discarding edges with appropriate probabilities one can go over to a
reduced oriented graph $R\subseteq R'$ which still inherits the minimum degree and density of~$G$.
\begin{lemma} \label{red-orient}
For every $\eps\in (0,1)$ and there exist numbers~$M'=M'(\eps)$ and $n_0=n_0(\eps)$
such that the following holds. Let $d\in [0,1]$ with $\eps \le d/2$,
let~$G$ be an oriented graph of order
$n \ge n_0$ and let~$R'$ be the reduced digraph with parameters $(\eps,d)$
obtained by applying the Diregularity Lemma to~$G$ with $M'$ as the lower bound on the number of clusters.
Then~$R'$ has a spanning oriented subgraph~$R$ such that
\begin{itemize}
\item[{\rm (a)}]$\delta^+(R)\ge (\delta^+(G)/|G|-(3\eps+d))|R|$,
\item[{\rm (b)}] $\delta^-(R)\ge(\delta^-(G)/|G|-(3\eps+d))|R|$,
\item[{\rm (c)}]
for all disjoint sets $S,T \subset V(R)$ with $e_G(S^*,T^*) \ge 3d n^2$ we have $e_R(S,T) > d|R|^2$,
where $S^*:= \bigcup_{i \in S} V_i$ and $T^*:= \bigcup_{i \in T} V_i$.
\item[{\rm (d)}]
for every set $S\subset V(R)$ with $e_G(S^*) \ge 3d n^2$ we have $e_R(S) > d|R|^2$,
where $S^*:= \bigcup_{i \in S} V_i$.
\end{itemize}
\end{lemma}

\begin{proof}
Statements (a) and (b) are proven in~\cite{KKO} by considering the following probabilistic argument.
Let~$R$ be the spanning oriented subgraph obtained from~$R'$ by deleting edges randomly as follows.
For every unordered pair~$V_i,V_j$ of clusters we delete the edge~$ij$ (if it lies in~$R'$)
with probability
\begin{equation}\label{eq:probdelete}
\frac{e_{G'}(V_j,V_i)}{e_{G'}(V_i,V_j)+e_{G'}(V_j,V_i)}.
\end{equation}
Otherwise we delete~$ji$ (if it lies in~$R'$). We interpret~(\ref{eq:probdelete})
as~0 if $ij,ji\notin E(R')$. So if~$R'$ contains at most one of the edges $ij,ji$
then we do nothing. We do this for all unordered pairs of clusters independently.
In~\cite{KKO} it was shown that $R$ satisfies (a) and (b) with probability at least $3/4$.
For~(c), note that
\begin{align*}
\mb{E} (e_R(S,T)) & = \sum_{i \in S, j \in T}{\frac{e_{G'}(V_i,V_j)}{e_{G'}(V_i,V_j)+e_{G'}(V_j,V_i)}}
\ge \sum_{i \in S, j \in T} \frac{e_{G'}(V_i,V_j)}{|V_i||V_j|} \\
 & \ge  \frac{e_G(S^*,T^*) - (d+\eps)n|S|m}{m^2} > \frac{3}{2}d|R|^2.
\end{align*}
A Chernoff-type bound (see e.g.~\cite[Cor.~A.14]{ProbMeth}) now implies that there exists
an absolute constant~$c$ such that%
     \COMMENT{To see that this Cor can be applied, note that $e_R(S,T)$ is a sum of independent indicator variables}
\begin{align*}
\mathbb{P}(e_R(S,T)<d|R|^2) & \le
\mathbb{P} \left( |e_R(S,T)-\mathbb{E}(e_R(S,T))| >\mathbb{E}(e_R(S,T)/3) \right)\\
& \le {\eul}^{-c \mathbb{E}(e_R(S,T))}\le {\eul}^{-c d|R|^2}.
\end{align*}
But the number of pairs $(S,T)$ of disjoint subsets of~$V(R)$ is at most $(2^{|R|})^2$,
so the probability that
$R$ does not satisfy (c) is at most  $2^{2|R|} {\eul}^{-c d|R|^2} < 1/4$. (Here we used
that $|R|\ge M'$ is sufficiently large.) The proof for~(d) is similar.
Altogether, this implies that there must be some outcome which satisfies (a)--(d).
\end{proof}


The oriented graph~$R$ given by Lemma~\ref{red-orient} is called the \emph{reduced
oriented graph with parameters $(\eps,d)$}. The spanning oriented subgraph~$G^*$ of
the pure digraph~$G'$ obtained by deleting all the $V_i$-$V_j$ edges whenever
$ij\in E(R')\setminus E(R)$ is called the \emph{pure oriented
graph}. Given an oriented subgraph $S\subseteq R$, the
\emph{oriented subgraph of~$G^*$ corresponding to~$S$}
is the oriented subgraph obtained from~$G^*$ by deleting all those vertices that lie in clusters not
belonging to~$S$ as well as deleting all the $V_i$-$V_j$ edges for all pairs $V_i,V_j$ with
$ij\notin E(S)$.

In our proof of Theorem~\ref{main} we will also need the Blow-up Lemma of Koml\'os, S\'ark\"ozy and
Szemer\'edi~\cite{KSSblowup}. It implies that dense regular pairs behave
like complete bipartite graphs with respect to containing bounded degree
graphs as subgraphs.

\begin{lemma}[Blow-up Lemma, Koml\'os, S\'ark\"ozy and
Szemer\'edi~\cite{KSSblowup}]\label{standardblowup}
Given a graph $F$ on $[k]$ and positive numbers $d,\Delta$, there is
a positive real $\eta_0=\eta_0(d,\Delta,k)$
such that the following holds for all positive numbers
$\ell_1,\dots,\ell_k$ and all $0<\eta\le \eta_0$. Let $F'$ be the graph obtained
from $F$ by replacing each vertex $i\in F$ with a set $V_i$ of $\ell_i$ new
vertices and joining
all vertices in $V_i$ to all vertices in $V_j$ whenever $ij$ is an edge
of $F$. Let $G'$ be a spanning subgraph of $F'$
such that for every edge $ij\in F$ the graph $(V_i,V_j)_{G'}$ is
$(\eta,d)$-super-regular. Then $G'$ contains a copy of every subgraph $H$ of
$F'$ with $\Delta(H)\le \Delta$. Moreover, this copy of $H$ in $G'$ maps the vertices of $H$
to the same sets $V_i$ as the copy of $H$ in~$F'$, i.e.~if $h \in V(H)$ is mapped to $V_i$ by the copy
of~$H$ in~$F'$, then it is also mapped to $V_i$ by the copy of $H$ in~$G'$.
\end{lemma}
The `moreover' part of the statement is not part of the usual statement of the Blow-up Lemma
but is stated explicitly in its proof.

We will also need to use the following stronger and more technical version due to
Csaba~\cite{csaba_06_bollobas_eldridge}. (The case when $\Delta=3$ of this is implicit
in~\cite{Csaba_Szemeredi_Bollobas_Eldridge_D3}.) It removes the dependency of the regularity
constant $\eta$ on the number $k$ of clusters. Also, it does not demand
super-regularity. The latter is replaced by conditions  (C6),(C8) and (C9) below.
Moreover, it is explicitly formulated to allow for a set $V_0$ of exceptional vertices in the
host graph $G'$ which are not part of any regular pairs.
For this to work, one first has to find a suitable partition $L_0,\dots,L_k$ with $|L_0|=|V_0|$
of the graph $H$ which we are aiming to embed
and a suitable bijection $\phi:L_0\to V_0$.
The embedding of~$H$ into~$G'$ guaranteed by the Blow-up Lemma is then found by a randomized
algorithm which first embeds each vertex $x\in L_0$ to~$\phi(x) \in V_0$
and then successively embeds the remaining vertices of~$H$.
Condition~(C1) requires that there are not too many exceptional vertices and~(C2) ensures
that we can embed the vertices in~$L_0$ without affecting the neighbourhood of other such vertices.
$L_i$ will be embedded into the cluster~$V_i$ of $G'$, so we clearly need to assume~(C3).
Condition~(C5) gives us a reasonably
large set~$D$ of `buffer vertices' which will be embedded last by the randomized algorithm.
(C7)~ensures that the exceptional vertices have large degree
in all `neighbouring clusters'.
(C8) and~(C9) allow us to embed those vertices whose set of candidate images
in~$G'$ has grown very small at some point of the algorithm.

In the statement of Lemma~\ref{blowup} and later on
we write $0<a_1 \ll a_2 \ll a_3$ to mean that we can choose the constants
$a_1,a_2,a_3$ from right to left. More
precisely, there are increasing functions $f$ and $g$ such that, given
$a_3$, whenever we choose some $a_2 \leq f(a_3)$ and $a_1 \leq g(a_2)$, all
calculations needed in the proof of Lemma~\ref{blowup} are valid.
Hierarchies with more constants are defined in the obvious way.%
      \COMMENT{The BL is trivial unless $m\ge 1/\eps$ (as otherwise each of the
non-empty bipartite graphs $(V_i,V_j)_{G'}$ must be complete) and this seems to be
enough for the pf of the BL, ie we don't need the additional assumptions
that $N\gg k$ or $N\gg 1/\eps$.}

\begin{lemma}[Blow-up Lemma, Csaba~\cite{csaba_06_bollobas_eldridge}]\label{blowup}
For all numbers~$\Delta, K_1,K_2,K_3$ and every positive constant~$c$ there exists an number~$N$
such that whenever $\eps,\eps',\delta',d$ are positive constants with
$$0<\eps\ll\eps'\ll\delta'\ll d\ll 1/\Delta,1/K_1,1/K_2,1/K_3,c
$$
the following holds. Suppose that~$G'$ is a graph of order~$n\ge N$ and $V_0,\dots,V_k$ is a partition
of~$V(G')$ such that the bipartite graph~$(V_i,V_j)_{G'}$ is $\eps$-regular
with density either $0$ or~$d$ for all $1\le i<j\le k$. Let~$H$ be a graph on~$n$ vertices with
$\Delta(H)\le \Delta$ and let $L_0, L_1,\dots, L_{k}$ be a partition of~$V(H)$
with $|L_i|=|V_i|=:m$ for every $i=1,\dots, k$. Furthermore, suppose that
there exists a bijection $\phi:L_0\rightarrow V_0$ and a set $I\subseteq V(H)$
of vertices at distance at least~$4$
from each other such that the following conditions hold:
\begin{itemize}
\item[\textup{(C1)}]$|L_0|=|V_0|\leq K_1dn$.
\item[\textup{(C2)}]$L_0\subseteq I$.
\item[\textup{(C3)}]$L_i$ is independent for every $i=1,\dots, k$.
\item[\textup{(C4)}]$|N_H(L_0)\cap L_i|\leq K_2dm$ for every $i=1,\dots, k$.
\item[\textup{(C5)}]For each $i=1,\dots, k$ there exists $D_i\subseteq I\cap L_i$ with
$|D_i|=\delta'm$ and such that for $D:=\bigcup_{i=1}^k D_i$ and all $1\leq i<j\leq k$
\[||N_H(D)\cap L_i|-|N_H(D)\cap L_j||<\eps m.\]
\item[\textup{(C6)}]If $xy\in E(H)$ and $x\in L_i,y\in L_j$ then $(V_i,V_j)_{G'}$ is $\eps$-regular
with density~$d$.
\item[\textup{(C7)}]If $xy\in E(H)$ and $x\in L_0, y\in L_j$ then $|N_{G'}(\phi(x))\cap V_j|\geq cm$.
\item[\textup{(C8)}]For each $i=1,\dots, k$, given any $E_i\subseteq V_i$ with $|E_i|\leq\eps' m$
there exists a set $F_i\subseteq(L_i\cap(I\setminus D))$ and a bijection $\phi_i:E_i\rightarrow F_i$
such that $|N_{G'}(v)\cap V_j|\geq (d-\eps)m$ whenever $N_H(\phi_i(v))\cap L_j\neq \emptyset$
(for all $v\in E_i$ and all $j=1,\dots,k$).
\item[\textup{(C9)}]Writing $F:=\bigcup_{i=1}^k F_i$ we have that
$|N_H(F)\cap L_i|\leq K_3\eps'm$.
\end{itemize}
Then~$G'$ contains a copy of~$H$ such that the image of~$L_i$ is~$V_i$ for all $i=1,\dots,k$ and the image of
each $x\in L_0$ is $\phi(x)\in V_0$.
\end{lemma}

The additional properties of the copy of~$H$ in~$G'$ are not included in the statement of the
lemma in~\cite{csaba_06_bollobas_eldridge} but are stated explicitly in the proof.

We aim to apply the Blow-up Lemma with~$G'$ being obtained from the underlying graph
of the pure oriented graph.
In order to satisfy~(C8), it will suffice to ensure that all the edges
of a suitable 1-factor in the reduced oriented graph~$R$ correspond to
$(\eps,d)$-super-regular pairs of clusters. A simple calculation implies
that this can be ensured by removing a small proportion
of vertices from each cluster~$V_i$, and so~(C8) will be satisfied.
However, (C6) requires all the edges of~$R$ to correspond to $\eps$-regular pairs of
density precisely~$d$ and not just at least~$d$.
(Although, as remarked by Csaba~\cite{csaba_06_bollobas_eldridge}, it actually suffices that
the densities are close to~$d$ in terms of~$\eps$.)
This is guaranteed by the second part of the following proposition from~\cite{KKO}.

\begin{prop}\label{superreg}
Let $M',n_0,D$ be positive numbers and let $\eps,d$ be positive reals such that
$1/n_0\ll 1/M'\ll \eps\ll d\ll 1/D$. Let~$G$ be an oriented graph of order at
least~$n_0$. Let~$R$ be the reduced oriented graph with parameters $(\eps,d)$
and let~$G^*$ be the pure oriented
graph obtained by successively applying first the Diregularity Lemma with $\eps$,
$d$ and~$M'$ to~$G$ and then Lemma~\ref{red-orient}. Let~$S$ be an oriented subgraph
of~$R$ with $\Delta(S)\le D$. Let~$G'$ be the underlying graph
of~$G^*$. Then one can delete $2D\eps|V_i|$ vertices from each cluster~$V_i$
to obtain subclusters $V'_i\subseteq V_i$ in such a way that~$G'$ contains a subgraph~$G'_S$ whose vertex
set is the union of all the $V'_i$ and such that
\begin{itemize}
\item $(V'_i,V'_j)_{G'_S}$ is $(\sqrt{\eps}, d-4D\eps)$-super-regular whenever~$ij\in E(S)$,
\item $(V'_i,V'_j)_{G'_S}$ is $\sqrt{\eps}$-regular and has density~$d-4D\eps$
whenever~$ij\in E(R)$.
\end{itemize}
\end{prop}

\section{The non-extremal argument}

This section covers the main part of the argument for the case when the original oriented graph
$G$ is not close to being the `extremal graph' described in Section~\ref{extremal}.
The following well-known fact will be very useful.

\begin{prop} \label{1-factor}
Suppose that~$G$ is a digraph such that $|N^+(S)| \ge |S|$ for
every $S \subset V(G)$. Then $G$ has a $1$-factor.
\end{prop}
\begin{proof}
The result follows immediately by applying Hall's matching
theorem to the following bipartite graph~$H$:
the vertex classes $A,B$ of~$H$ are both copies of the vertex set of the original digraph $G$ and
we connect a vertex $a \in A$ to $b \in B$ in~$H$ if there is a directed edge from $a$ to $b$
in~$G$. A perfect matching in~$H$ corresponds to a 1-factor in~$G$.
\end{proof}

\begin{lemma} \label{find-cycle} Suppose $1/n \ll
\eps_0 \ll 1/k \ll \eps \ll d \ll c \ll 1$, and that~$R$ is an oriented graph on $[k]$ such that
\begin{itemize}
\item $\delta^0(R) \ge 2ck$,
\item $|N^+_R(S)| \ge |S| + ck$ for any $S \subset [k]$ with $|S| \le (1-c)k$, and
\item ${\mathcal C}$ is a $1$-factor in $R$.
\end{itemize}
Also, suppose that
$R^*$ is an oriented graph obtained from $R$ by adding a set
$U_0$ of at most $\eps_0 n$ vertices and some edges so that
every $x \in U_0$ has both an inneighbour and an outneighbour in $[k]$.
Suppose $G^*$ is an oriented graph on $n$ vertices with vertex partition
$U_0 , U_1 , \dots , U_k$, such that
\begin{itemize}
\item[(i)] each $U_i$ is independent; 
\item[(ii)] if $ij$ is an edge on some cycle from ${\mathcal C}$, then the bipartite
graph $(U_i,U_j)$ consisting of all the $U_i$-$U_j$ edges in~$G^*$ forms an
$(\eps,d)$-super-regular pair,
while if $ij \in E(R)$, the bipartite graph $(U_i,U_j)$ forms an
$\eps$-regular pair of density $d$ in $G^*$, and finally $(U_i,U_j)$ is empty whenever
$ij\notin E(R)$;
\item[(iii)] for all $x \in U_0$ and $i \in [k]$, we have
$|N^+_{G^*}(x) \cap U_i| > c|U_i|$ whenever $xi \in E(R^*)$ and
$|N^-_{G^*}(x) \cap U_i| > c|U_i|$ whenever $ix \in E(R^*)$;
\item[(iv)] $|U_1|=\dots=|U_k|=:m$.
\end{itemize}
Then $G^*$ has a Hamilton cycle.
\end{lemma}

\begin{proof}
Throughout the proof of the lemma, given a cycle $C\in {\mathcal C}$ and a vertex $x \in C$,
we will denote the predecessor of $x$ on $C$ by $x^-$ and its successor by $x^+$.
We extend this convention
to sets of vertices in the natural way. We say that a walk $W$ in~$R^*$ is \emph{balanced} with
respect to ${\mathcal C}$ if for each cycle $C\in {\mathcal C}$ every vertex on $C$
is visited the same number of times by~$W$.
We will first construct a closed walk $W$ in $R^*$ such that
\begin{itemize}
\item[(a)] $W$ is balanced with respect to ${\mathcal C}$;
\item[(b)] $W$ visits every vertex in $[k]$ at least once and at most $\sqrt{\eps_0} n$ times;
\item[(c)] $W$ visits every vertex in $U_0$ exactly once;
\item[(d)] any two vertices in $U_0$ are at distance at least $4$ along $W$.
\end{itemize}
To achieve the balance property (a), we will build up $W$ from certain special walks
(which are themselves balanced with respect to ${\mathcal C}$):
a \emph{${\mathcal C}$-shifted walk joining $x$ to $y$ and traversing~$t$ cycles from~$\C$}
is%
       \COMMENT{changed def since length of a walk is usually the number of edges}
a walk of the form
$a_1C_1a_1^-a_2C_2a_2^- \dots a_t C_t a_t^-$, where $x=a_1$, $y=a_t^-$ and
the $C_i$ are (not necessarily distinct) cycles from ${\mathcal C}$.
(Note that this definition is slightly different from the corresponding one in~\cite{KKO}.)

We claim that for any two vertices
$x,y \in [k]$ there is a ${\mathcal C}$-shifted walk in $R$ joining $x$ to~$y$
which traverses at most $c^{-1}+1$ cycles from~$\C$.
To see this, let $S_i \subset [k]$ be the set of vertices
that can be reached from $x$ by a $\C$-shifted walk traversing at most~$i$ cycles from~$\C$.
Then $S_1=\{x^-\}$ and $S_{i+1}= (N^+_R(S_i))^-$ for $i \ge 1$.
By assumption on $R$,  if $|S_i| \le (1-c)k$ then we have
$|S_{i+1}| \ge |S_i| + ck$. So if $i \ge c^{-1}$ we have
$|S_i| \ge (1-c)k$. The minimum semi-degree condition on $R$ implies that
$|N^-_R(y^+)| \ge 2ck$.
So there is a vertex $z \in S_i \cap N^-_R(y^+)$ and we can
reach $y$ by following a $\C$-shifted walk to~$z$ which traverses at most~$i$ cycles
from~$\C$ and then
adding $zy^+Cy$ (where~$C$ is the cycle from $\mc{C}$ containing~$y$.)

For each pair $x,y$ let $W(x,y)$ denote such a $\C$-shifted walk.
To construct~$W$, we start at vertex $1 \in [k]$ and follow the
walks%
    \COMMENT{Previous def wasn't quite working}
$W(1,2^-), 2^-2,W(2,3^-),3^-3, \dots, W(k-1,k^-)$. Note that altogether these walks yield a
single (balanced) ${\mathcal C}$-shifted walk joining $1$ to $k^-$ which visits every vertex of
$R$ at least once. Next we aim to extend this walk to incorporate the vertices in $U_0$.
For this, suppose that $U_0=\{v_1,\dots,v_q\}$, say.
For each $i \in [q]$, we now select vertices
$w_i \in N^+_{R^*}(v_i)$, $u_i \in N^-_{R^*}(v_i)$%
     \COMMENT{Complete Disjointness is not achievable and not needed. Requiring
$w_i \neq u_{i+1}$ may make things easier for the reader though}
and follow the walks $W(k,u_1)$, $u_1 v_1 w_1$,
$W(w_1,u_2)$, $u_2 v_2 w_2$, $\dots$, $W(w_{q-1},u_q)$,
$u_q v_q w_q$. Finally we close the walk by following $W(w_q,1^-)$ and the edge~$1^-1$.
Note that the resulting closed walk
$W$ is balanced with respect to~$\C$ since this holds for each of the walks~$W(*,*)$.
Moreover,~$W$ traverses at least one cycle of ${\mathcal C}$ between each visit to $U_0$ (by definition),
so any two vertices in $U_0$ are at distance at least $4$ along $W$. Every vertex in
$U_0$ is visited exactly once, and vertices in $[k]$ are visited at most
$(k+|U_0|)(c^{-1}+1) < \sqrt{\eps_0} n$ times, so $W$ has the required properties (a)--(d).

We will use the Blow-up Lemma (Lemma~\ref{blowup}) to `transform' $W$ into a
Hamilton cycle of $G^*$. As a preliminary step, it will be useful to consider an
auxiliary (undirected) graph $K^*$ which it is obtained from $G^*$ by replacing every
$\eps$-regular pair $(U_i,U_j)$ of nonzero density (i.e.~those which correspond to edges of $R$)
by a complete bipartite graph. The neighbourhoods of the
vertices in $U_0$ are not affected.
We can find a Hamilton cycle $C_{{\rm Ham}}$ in $K^*$ as follows.
First we find a one-to-one mapping of $W$ to a cycle $W'$ in $K^*$ such that
each vertex of $U_0$ is mapped to itself and for each $i \in [k]$
visits of $W$ to a vertex $i\in R$ are mapped to distinct vertices in $U_i$;
to do this we first greedily select distinct images for $u_i$ and $w_i$ for all $i \in [q]$,
which is possible by assumption (iii) and the inequality
$|U_0|\le \eps_0 n \le cn/(5k) \le c|U_i|/4$. We can then complete the rest of the mapping
arbitrarily since~$W$ visits every vertex $i\in R$ at most $\sqrt{\eps_0}n$ times.
Next we will extend $W'$ to a Hamilton cycle.
For each cycle $C \in {\mathcal C}$, let $m_C$ denote the number of times that
$W$ visits a vertex $c\in C$ (note that this number is the same for all vertices $c\in C$).
Fix one particular occasion on which $W$ `winds around' $C$ and replace the corresponding
path in $W'$ by a path $P_C$ with the same endpoints that winds around the
`blow-up' $\bigcup_{c \in C} U_c$ of~$C$ in
$K^*$ exactly $m-m_C+1$ times, so that it exhausts all the vertices in $\bigcup_{c \in C} U_c$.
Since $\mc{C}$ is a $1$-factor the resulting cycle
$C_{{\rm Ham}}$ uses all vertices in~$K^*$, i.e.~it is Hamiltonian.
For future reference we note that
the number of times that $P_C$ winds around $\bigcup_{c \in C} U_c$ is
$m-m_C+1 \ge m -\sqrt{\eps_0} n \ge m - \sqrt{\eps_0} 2mk \ge (1-\eps)m$.
(Here we used property~(b) and $\eps_0 \ll 1/k,\eps$.)

Now we use the version of the Blow-up Lemma by Csaba (Lemma~\ref{blowup})
to show that~$C_{{\rm Ham}}$ corresponds to a Hamilton cycle in $G^*$.
We fix additional constants $\eps',\delta'$ with $\eps \ll \eps' \ll \delta' \ll d$.
We will apply the Blow-up Lemma with~$H$
being the underlying graph of $C_{{\rm Ham}}$ and
$G'$ being obtained from the underlying graph of $G^*$
as follows: if $xi \in E(R^*)$ with $x \in U_0$ and $i \in [k]$,
then delete all those edges which correspond to edges oriented from $U_i$ to $x$;
similarly, if $ix \in E(R^*)$ with $x \in U_0$ and $i \in [k]$,
then delete all those edges which correspond to edges oriented from $x$ to $U_i$.
This deletion ensures that when we embed $C_{{\rm Ham}}$ as an undirected cycle in $G'$ the directions
of the available edges will in fact make it a directed cycle in $G^*$.
For all $i \ge 0$, $U_i$ will play the role
of $V_i$ in the Blow-up Lemma and we take~$L_0,L_1,\dots,L_k$ to
be the partition of~$H$ induced by~$V_0,V_1,\dots,V_k$. $\phi: L_0\to V_0$ will be the
obvious bijection (i.e.~the identity). To define the set $I\subseteq V(H)$ of vertices
of distance at least~$4$ from each other which is used in the Blow-up Lemma, for each
$C \in \mc{C}$ let $P'_C$ be the subpath of~$H$ corresponding to the path $P_C$
(defined in the construction of $C_{{\rm Ham}}$).
Note that $|P'_C|=|P_C| \ge (1-\eps)m|C|$.
For each $i=1,\dots,k$, let $C_i\in \C$ denote the cycle containing~$i$ and let
$J_i\subseteq L_i$ consist of all those vertices in~$L_i\cap V(P'_{C_i})$
which have distance at least~4
from the endvertices of~$P'_{C_i}$. Thus in the graph~$H$
each vertex $u\in J_i$ has one of its neighbours in the set~$L^-_i$
corresponding to the predecessor
of~$i$ on~$C_i$ and its other neighbour in the set~$L^+_i$
corresponding to the successor of~$i$ on~$C_i$.
Moreover, all the vertices in~$J_i$ have distance at least~4 from all the vertices in~$L_0$
and the lower bound on  $|P_C|$ implies that $|J_i| \ge 9m/10$.
It is easy to see that one can greedily choose a set $I_i\subseteq J_i$
 of size $m/10$ such that the vertices in
$\bigcup_{i=1}^k I_i$ have distance at least~4 from each other.%
     \COMMENT{Indeed, suppose that we are about to choose~$I_i$.
For each of the at most 6 clusters~$V_j$
which lie on~$C_i$ and have distance at most~3 from~$i$ on~$C_i$ the set~$I_j$ forbids at most
$|I_j|$ vertices in~$J_i$ (if~$I_j$ was chosen before~$I_i$). This leaves us with
$9m/10-6m/10=3m/10$ vertices from~$J_i$. As the vertices in~$J_i$ have distance at least~3
from each other, we can choose every second of the remaining vs in~$J_i$, ie get a set
of at least $3m/20>m/10$ vertices.}
Let $I:=L_0\cup \bigcup_{i=1}^k I_i$.

Let us now check conditions~(C1)--(C9). (C1) holds with~$K_1:=1$ since
$|L_0|=|U_0|\le \eps_0 n\le dn$. (C2) holds by definition of~$I$.
(C3) holds since~$H$ is a Hamilton cycle in~$K^*$ and $K^*$ inherits the independence
property (i) in the statement of the lemma from $G^*$. This also implies that for every edge $xy\in H$
with $x\in L_i,y\in L_j$ ($i,j\ge 1$) we must have that $ij\in E(R)$
(provided that the direction of~$xy$ was chosen such that $xy\in C_{{\rm Ham}}$).
Thus~(C6) holds
as every edge of~$R$ corresponds to an $\eps$-regular pair of clusters having density~$d$.
(C4) holds with $K_2:=1$ because
\[|N_H(L_0)\cap L_i|\leq 2|L_0|= 2|U_0| \le 2\eps_0 n \le dm.
\]
(The final inequality follows from the fact that $\eps_0 \ll 1/k,d$.)
For~(C5) we need to find a set $D\subseteq I$ of buffer vertices. Pick
any set $D_i\subseteq I_i$ with $|D_i|=\delta' m$ and let $D:=\bigcup_{i=1}^k D_i$.
Since $I_i\subseteq J_i$ we have that $|N_H(D)\cap L_j|=2\delta'm$
for all $j=1,\dots,k$. Hence
$$||N_H(D)\cap L_i|-|N_H(D)\cap L_j||=0
$$
for all $1\leq i<j\leq k$ and so~(C5) holds. (C7) holds by assumption (iii) on $R^*$
and the fact that our walk $W$ used only edges of $R^*$.

(C8) and~(C9) are now the only conditions we need to check. Given a set $E_i\subseteq V_i=U_i$
of size at most~$\eps' m$, we wish to find
$F_i\subseteq (L_i\cap(I\setminus D))=I_i\setminus D$ and a
bijection $\phi_i:E_i\rightarrow F_i$ such that every $v\in E_i$ has a large number
of neighbours in every cluster~$V_j$ for which~$L_j$ contains a neighbour of~$\phi_i(v)$. Pick any set
$F_i\subseteq I_i\setminus D$ of size~$|E_i|$. (This can be done since
$|D\cap I_i|=\delta' m$ and so $|I_i\setminus D|\ge
m/10-\delta' m\gg \eps' m$.) Let $\phi_i:E_i\rightarrow F_i$ be an arbitrary bijection.
To see that~(C8) holds with these choices, consider any vertex $v\in E_i\subseteq V_i$
and let~$j$ be such that~$L_j$ contains a neighbour of~$\phi_i(v)$ in~$H$. Since
$\phi_i(v)\in F_i\subseteq I_i\subseteq J_i$, this means that~$j$ must be a neighbour
of~$i$ on the cycle $C_i\in \C$ containing~$i$. But this implies that
$|N_{G'}(v)\cap V_j|\ge (d-\eps)m$ since each edge of the union~$\bigcup_{C\in \C} C\subseteq R$
of all the cycles from~$\C$ corresponds to an $(\eps, d)$-super-regular pair in~$G'$.

Finally, writing $F:= \bigcup_{i=1}^k F_i$ we have
\[|N_H(F)\cap L_i|\le 2\eps'm\]
(since $F_j\subseteq J_j$ for each $j=1,\dots,k$) and so~(C9) is satisfied with $K_3:=2$.
Hence (C1)--(C9) hold and so we can apply the Blow-up Lemma to obtain a Hamilton cycle
in~$G'$ such that the image of~$L_i$ is~$V_i$ for all $i=1,\dots,k$
and the image of each $x\in L_0$ is $\phi(x)\in V_0=U_0$. By construction of $G'$
this undirected Hamilton cycle corresponds to a directed Hamilton cycle in~$G^*$.
\end{proof}

\section{Proof of Theorem~\ref{main}}

Define $M',M'_0\in \mathbb{N}$ and additional constants so that
$$1/n_0 \ll 1/M'_0\ll \eps_0 \ll d_0\ll 1/M'\ll \eps \ll d \ll c \ll \eta \ll 1.$$
Suppose $G$ is an oriented graph on $n \ge n_0$ vertices with minimum
semi-degree $\delta^0(G) \ge \frac{3n-4}{8}$ and no Hamilton cycle.
Apply the Diregularity Lemma (Lemma~\ref{direg}) to $G$ with parameters $\eps^2/3, d,M'$
to obtain a partition $V_0,V_1, \dots, V_k$ of $V(G)$ with $k\ge M'$.
Let~$R$ be the reduced oriented graph with parameters $(\eps^2/3, d)$ given by
applying Lemma~\ref{red-orient} and let~$G^*$ be the pure oriented graph.%
     \COMMENT{In the proof of Lemma~\ref{cyclelemma} $G^*$ will be the subgraph of
the pure digraph corresponding to~$R$. The $\eps$ in Lemma~\ref{cyclelemma} will
play the role of $\sqrt{3\eps}$ here.}
Lemma~\ref{red-orient}(a) and (b) implies
\begin{equation}\label{eq:minR}
\delta^0(R) > (3/8 - 1/(2n) - d - \eps^2)k > (3/8-2d)k.
\end{equation}

We divide the remainder of the argument into two cases, according to
whether there is a set%
      \COMMENT{Actually, the $2cn$ (instead of just $cn$) in Case~1 looks a bit funny now,
but by leaving it as it is we don't have to modify and recheck numbers in Case~2.}
$S \subset [k]$ with $k/3 < |S| < 2k/3$ and
$|N^+_R(S)| < |S| + 2ck$.

\medskip

\nib{Case 1.} \emph{$|N^+_R(S)| \ge |S| + 2ck$ for every $S \subset [k]$
with $k/3 < |S| < 2k/3$.}

\medskip

Note that if $0 \ne |S| \le k/3$ then $|N^+_R(S)| \ge \delta^+(R) > |S|+d^2 k$ (say), and if
$|S| > 2k/3$ then $|S| + |N^-_R(i)| > k$ for any $i \in [k]$, i.e.
$S \cap N^-_R(i) \ne \emptyset$, and $N^+_R(S)=[k]$.
So $|N^+_R(S)| \ge |S|$ for every $S \subset [k]$, and
Proposition~\ref{1-factor} implies that there is a
$1$-factor ${\mathcal C}$ in $R$. In fact,%
      \COMMENT{In the proof of Lemma~\ref{cyclelemma} all of this still works since
$d\ll \nu\ll \tau\ll \eta$.}
\begin{equation}\label{eq:expS}
|N^+_R(S)| \ge |S| + d^2 k
\end{equation}
for every $S \subset [k]$ with $d^2k < |S| < (1-d^2)k$.

We would now like to apply Lemma~\ref{find-cycle} to $R$, ${\mathcal C}$ and $G^*$
to find a Hamilton cycle. However, the size of the exceptional set $V_0$ may be larger
than the number of exceptional vertices allowed for by the lemma.
To deal with this problem, we first partition $V(G)\setminus V_0$ into sets $A$ and $B$.
Then we further refine the above `regularity partition' within $A \cup V_0$ to obtain a very small
`exceptional set' $A_0''$. It then turns out that we can use Lemma~\ref{find-cycle}
to find a
Hamilton path in $(A \cup V_0) \setminus A_0''$ and a Hamilton cycle in $B \cup A_0''$
which together form a Hamilton cycle in $G^*$.

By adding at most $k$ vertices to $V_0$ we may assume that each $V_i$ contains
an even number of vertices.

\medskip

\nib{Claim 1.1.} \emph{There is a partition of $V \sm V_0$ into sets $A$ and $B$
which has the following properties:%
    \COMMENT{need that the $|A_i|$ have the same size for Prop.~\ref{superreg}.
Simply discarding vertices into $V_0$ would be simplest solution but would make statement
of Claim 1.1 uglier}
\begin{itemize}
\item $|A_i|, |B_i| = \frac{1}{2}|V_i|$, where we write $A_i := V_i \cap A$ and
$B_i := V_i \cap B$ for every $i \in [k]$;
\item $|N^+_G(x) \cap A_i|, |N^+_G(x) \cap B_i| = \frac{1}{2}|N^+_G(x) \cap V_i| \pm n^{2/3}$ for every
vertex $x\in G$; and similarly for $N^-_G(x)$;
\item $R$ is the oriented reduced graph with parameters $(\eps^2,3d/4)$ corresponding
to the partition $A_1,\dots,A_k$ of the vertex set of~$G^*[A]$; this also holds for the
partition $B_1,\dots,B_k$.
\end{itemize}}

\begin{proof}
For each cluster~$V_i$ consider a random partition of~$V_i$ into two sets~$A_i$ and~$B_i$
obtained by assigning a vertex $x\in V_i$ to~$A_i$ with probability~$1/2$ independently
of all other vertices in~$V_i$. So $A:=\bigcup_{i=1}^k A_i$ and $B:=\bigcup_{i=1}^k B_i$.
Then the probability that $|A_i|=|B_i|$ is at least $1/(3\sqrt{|V_i|})$
(see e.g.~\cite[p.~6]{BollobasRG}).
Also, standard Chernoff bounds imply that the probability that there is a
vertex $x\in G$ whose degree in~$A_i$ is too large or too
small is exponentially small in~$|V_i|$. So with non-zero probability we obtain a
partition as desired.
To see the third property, note that the definition of regularity implies that
the pair $(A_i,A_j)_{G^*}$ consisting of all the $A_i$-$A_j$ edges in~$G^*$ is $\eps^2$-regular
and has density at least $3d/4$ whenever $ij\in E(R)$. On the other hand, $(A_i,A_j)_{G^*}$ is
empty whenever $ij\notin E(R)$ since $(V_i,V_j)_{G^*}\supset (A_i,A_j)_{G^*}$ is empty in this
case.
\end{proof}

\medskip

Apply Proposition~\ref{superreg}, with  $G^*[B]$ playing the role of $G^*$,
${\mathcal C}$ playing the role of $S$ (so $D=2$), the $B_i$ playing the role of the $V_i$ and
$\eps,d$ replaced by $\eps^2,3d/4$. We obtain an (oriented) subgraph $G_B'$ of $G^*[B]$
and sets $B_i'\subseteq B_i$ of size $4\eps^2|B_i|$ so that a pair $(B_i',B_j')_{G'_B}$
is $\eps$-regular of density $3d/4-8\eps^2 \ge d/2$
whenever $ij \in E(R)$, and so that $(B_i',B_j')_{G'_B}$ is $(\eps,d/2)$-super-regular%
     \COMMENT{Get $(\eps,3d/4-8\eps^2)$-super-regularity from Proposition~\ref{superreg}, but this
implies $(\eps,d/2)$-super-regularity}
whenever $ij$ is an edge of~$\C$.
Let $B':=B_1'  \cup \dots \cup B_k'$. So $V(G'_B)=B'$. Let $V_0'$ be obtained from
$V_0$ by adding the vertices in $B_i \setminus B_i'$ for all $i\in [k]$. Since
$|V_0| \le (\eps^2/3)n+k$ and $|B_i \sm B_i'| \le 4\eps^2|B_i|$ we have
$$|V_0'| < \eps n.$$

\medskip

\nib{Claim 1.2.} \emph{Each set $A_i$ ($i\in [k]$) contains a set~$U_i$ of size $2\eps |A_i|$
such that $\delta^0(G^*[A\setminus U])\ge dn/8$, where%
      \COMMENT{In the proof of Lemma~\ref{cyclelemma} we get $\delta^0(G^*[A\setminus U])\ge d\eta n/3$.}
$U:=\bigcup_{i=1}^k U_i$.}

\begin{proof}
Consider any cluster~$A_i$. Claim~1.1 implies that for all $j\in N^+_R(i)$ at most
$\eps^2 |A_i|$ vertices in~$A_i$
have less than $d|A_j|/2$ outneighbours in~$A_j$. Call these vertices \emph{useless for~$j$}.
So on average any vertex of~$A_i$ is useless for at most $\eps^2|N^+_R(i)|$ indices $j\in N^+_R(i)$.
Thus at most $\eps |A_i|$ vertices in~$A_i$ are useless for more than $\eps |N^+_R(i)|$ indices
$j\in N^+_R(i)$. Let $U'_i\subseteq A_i$ be a set of size $\eps |A_i|$ which contains all
these vertices and some extra vertices from~$A_i$ if necessary. Similarly, there is
a set $U''_i\subseteq A_i\setminus U'_i$ of size $\eps |A_i|$ such that every vertex
$v\in A_i\setminus U''_i$ has at least $d|A_j|/2$ inneighbours in~$A_j$ for all but at most
$\eps |N^-_R(i)|$ indices $j\in N^-_R(i)$. Let $U_i:=U'_i\cup U''_i$ and $U:=\bigcup_{i=1}^k U_i$.
Thus $|U|=2\eps |A|$ and every vertex $v\in A_i\setminus U_i$ satisfies%
     \COMMENT{In the proof of Lemma~\ref{cyclelemma} we get $d\eta n/3$ instead of $dn/8$.}
$$
d^+_{G^*[A\setminus U]}(v)\ge \frac{d}{2}(1-\eps)|N^+_R(i)||A_i|-|U|
\stackrel{(\ref{eq:minR})}{\ge} \frac{dn}{8}.
$$
The same bound holds for $d^-_{G^*[A\setminus U]}(v)$ and thus
$\delta^0(G^*[A\setminus U])\ge dn/8$.
\end{proof}

\medskip

Add all the vertices in~$U$ to~$V'_0$ to obtain a new set~$V''_0$ with
$$|V''_0|<3\eps n.
$$
Let~$G'_A$ be the subgraph of $G$ obtained from~$G^*[A\setminus U]$ by adding the set~$V''_0$
and setting the out- and inneighbourhoods of any $v\in V''_0$ in~$G'_A$ equal
to $N^+_{G}(v)\cap (A\setminus U)$ and $N^-_{G}(v)\cap (A\setminus U)$. Thus%
     \COMMENT{In the proof of Lemma~\ref{cyclelemma} we get $\delta^0(G'_A)\ge d\eta n/3$.}
$\delta^0(G'_A)\ge d|G'_A|/8$ by Claims~1.1 and~1.2.
In what follows, we will still write~$A_i$ for $A_i\setminus U_i$
and $A$ for $A\setminus U$. Then in~$G'_A$ each pair $(A_i,A_j)_{G'_A}$ with $ij\in E(R)$
is still $\eps$-regular and has density at least~$d/2$.

Apply the Diregularity Lemma (Lemma~\ref{direg}) to~$G'_A$ with
parameters $\eps_0^2$, $2d_0$, $M'_0$ and initial partition $V_0'',  A_1, \dots ,A_k$
to obtain a reduced digraph~$R'_A$ and a partition $A_0', A_1', \dots, A_\ell'$
(where the exceptional set is $A_0'$) which refines $V_0'',  A_1, \dots ,A_k$.
Apply Lemma~\ref{red-orient} to~$G'_A$ (playing the role of~$G$), $R'_A$ (playing the role
of~$R'$) and $\eps_0^2,2d_0,M'_0$ (playing the roles of~$\eps,d,M'$).
We obtain a reduced oriented graph~$R_A$ with parameters $(\eps^2_0,2d_0)$
such that%
      \COMMENT{In the proof of Lemma~\ref{cyclelemma} get $\delta^0(R_A) \ge d\eta \ell/4$.}
$$\delta^0(R_A) \ge d\ell/8-(3\eps_0^2+2d_0)\ell\ge d\ell/10$$
(this is what we need Claim~1.2 for).
Let~$G''_A$ denote the corresponding pure oriented graph.

\medskip

\nib{Claim 1.3.}  \emph{For any $T \subset [\ell]$ with $2d^2\ell < |T| < (1-2d^2)\ell$
we have $|N^+_{R_A}(T)| \ge |T| + d^2\ell/2$.}

\begin{proof}
Let $m := |A_i| = |A|/k$, for $1 \le i \le k$, and $m':=|A_x'|$, for $1 \le x \le \ell$.
Then $|A| + |V_0''| = \ell m' + |A_0'|$. Since $|V_0''| \le 3\eps n$ we get
$m' \le (|A| + 3\eps n)/\ell \le (1+7\eps)|A|/\ell$, and since
$|A_0'| \le \eps_0^2(|A|+|V_0''|) < \eps_0^2 n$
we get $m' \ge (|A| - \eps_0^2 n)/\ell \ge (1-3\eps_0^2)|A|/\ell$.
For $i \in [k]$ let $L_i := \{x \in [\ell] \colon A_x' \subset A_i \}$.
Then
$$
|L_i| \le \frac{m}{m'} \le \frac{km+|V_0''|}{km'} = \frac{\ell m' + |A_0'|}{km'}
 < \frac{\ell}{k} \left(1 + \frac{\eps_0^2 n}{(1-3\eps_0^2)|A|} \right) < (1+3\eps_0^2)\ell/k.
$$
Let $Z := \{i \in [k]: |L_i| < (1-10\eps)\ell/k\}$.
We will show that $Z$ is comparatively small. For this,
note that if $i \in Z$, then
$$
|A_i \setminus A_0'| = |L_i|m' \le (1-10\eps)(\ell/k)(1+7\eps)|A|/\ell
\le (1-\eps)|A|/k.
$$
Since $|A_i|=m=|A|/k$, this in turn shows that $|A_i \cap A_0'|\ge \eps |A|/k$.
So $|Z| \ge \eps k$ would imply that $|A_0'| \ge |Z| \eps |A|/k \ge \eps^2 |A| > \eps_0^2 n$,
contradiction. Thus we must have $|Z| \le \eps k$.

Let $T_i := T \cap L_i$, $S := \{i \in [k] \colon  |T_i| > d^3\ell/k \}$
and  $T' := \bigcup_{i \in S} T_i$.
Then $|T_i| \le |L_i| \le (1+3\eps_0^2)\ell/k$ and
$|T \sm T'| \le (k-|S|)d^3\ell/k+|\{x\in [\ell]: A'_x\subset V_0''\}|
< (d^3+7\eps) \ell$
imply that
\begin{equation}\label{eq:sizeS}
|S|\ge \frac{|T'|}{(1+3\eps_0^2)\ell/k} > |T|k/\ell - 2d^3k.
\end{equation}
Therefore $|S| > d^2k$ and so~(\ref{eq:expS}) implies that either $|S|\ge (1- d^2)k$ or
$|N^+_R(S)| \ge |S|+d^2k$. Suppose first that the latter holds. We will now show
that for all $i \in S$ and all $j \in N^+_R(i)$ we have
\begin{equation}\label{eq:Lj}
|L_j\sm  N^+_{R_A}(T)| < d^3\ell/k.
\end{equation}
Indeed, if not then $W_i:=\bigcup_{x \in T_i} A_x' \subset A_i$ and
$W_j:=\bigcup_{x \in L_j \sm N^+_{R_A}(T)} A_x' \subset A_j$ would span a bipartite
subgraph~$(W_i,W_j)_{G'_A}$ of $(A_i,A_j)_{G'_A}$ with vertex class sizes
$|W_i|,|W_j|\ge |A'_x|d^3\ell/k> \eps m$. Note that $e_{R_A}(T_i, L_j \sm N^+_{R_A}(T))=0$.
So Lemma~\ref{red-orient}~(c) applied
to~$T_i$ (playing the role of~$S$) and $L_j \sm N^+_{R_A}(T)$ (playing the role of~$T$)
shows that
$$e_{G'_A}(W_i,W_j)<6d_0 |G'_A|^2\le 6d_0 n^2\le 7d_0 (mk)^2\le \frac{7d_0 k^2}{\eps^2} |W_i||W_j|
\le \frac{d}{4} |W_i||W_j|,
$$
where the last inequality holds since $d_0\ll \eps$ and thus also $d_0\ll 1/k$.
This contradicts the fact that $(A_i,A_j)_{G'_A}$ is an $\eps$-regular pair of density at
least $d/2$, which proves~(\ref{eq:Lj}). Therefore
\begin{align*}
|N^+_{R_A}(T)| & \stackrel{(\ref{eq:Lj})}{>} \sum_{j \in N^+_R(S)} (|L_j| - d^3\ell/k)
>(|N^+_R(S)|-|Z|)(1-2d^3)\ell/k\\
& >(|S|+d^2k-\eps k)(1-2d^3)\ell/k \stackrel{(\ref{eq:sizeS})}{>} |T| + d^2\ell/2.
\end{align*}
The argument in the case when $|S| \ge (1-d^2)k$ is almost the same:
as mentioned at the beginning of Case~1 we now have $N^+_R(S) = [k]$ and so
$|N^+_{R_A}(T)| > (k-\eps k)(1-2d^3)\ell/k > (1-3d^3)\ell > |T| + d^2\ell/2$. \end{proof}

\medskip

Now the same argument that we used to find the $1$-factor $\mc{C}$ in $R$ gives
us%
     \COMMENT{In the proof of Lemma~\ref{cyclelemma} this still works since
$\delta^0(R_A) \ge d\eta \ell/4$ and so~$R_A$ is still expanding as $d\eta\gg d^2$.}
a $1$-factor $\mc{C}_A$ in $R_A$. Apply Proposition~\ref{superreg}, with
$G''_A$ playing the role of $G^*$,
$\mc{C}_A$ playing the role of~$S$ (so $D=2$), the $A_i'$ playing the role of the $V_i$ and
$\eps,d$ replaced by $\eps_0^2,2d_0$. We obtain an (oriented) subgraph $G_A'''$ of $G''_A$
and sets $A_i''\subseteq A'_i$ so that a pair $(A_i'',A_j'')_{G_A'''}$ is $\eps_0$-regular of density
$2d_0-8\eps_0^2 \ge d_0$ whenever $ij \in E(R_A)$, and so that $(A_i'',A_j'')_{G_A'''}$ is
$(\eps_0,d_0)$-super-regular whenever $ij$ is an edge
of $\C_A$. Let $A' := A_1'' \cup \cdots \cup A_\ell''$ and let~$A_0''$ be the set obtained from~$A_0'$
by adding all the vertices in $A'_i\setminus A''_i$ for all $i\in [\ell]$. Then $V(G_A''')=A'$
and
$$
|A_0''| < \eps_0^2 n + 4\eps_0^2 n < \eps_0 n.
$$

Let~$G^*_B$ be the subgraph of $G$ obtained from~$G'_B$ by adding the set~$A''_0$
and setting the out- and inneighbourhoods of any $v\in A''_0$ in~$G^*_B$ equal
to $N^+_{G}(v)\cap B'$ and $N^-_{G}(v)\cap B'$. We will apply Lemma~\ref{find-cycle} to
find a Hamilton cycle in $G_B^*$. For this, note that the reduced oriented graph $R$ defined above
satisfies the initial two conditions of Lemma~\ref{find-cycle} (c.f.~the discussion just before the
statement of Claim~1.1). Moreover, we will let the vertex partition $B_1',\dots,B_k'$
of $G_B^*-A_0''$ play the role of $U_1,\dots,U_k$ in the lemma.%
      \COMMENT{In the proof of Lemma~\ref{cyclelemma} $\nu$ will play the role of~$c$ in
Lemma~\ref{find-cycle}.}
The role of $U_0$ will be played by $A_0''$.
Now define an oriented graph $R^*_B$ with vertex
set $[k] \cup A_0''$ as follows. The restriction of $R^*_B$ to
$[k]$ is $R$. For each $x \in A_0''$ add an edge from $x$ to
$i \in [k]$ if $|N^+_{G}(x) \cap B'_i| > c|B'_i|$ and%
     \COMMENT{In the proof of Lemma~\ref{cyclelemma} have $>\nu |B'_i|$ instead of $c|B'_i|$.}
an edge from $i \in [k]$ to $x$ if $|N^-_{G}(x) \cap B'_i| > c|B'_i|$.
Using Claim~1.1, it is easy to see that
$x$ has indegree and outdegree at least%
    \COMMENT{Replaced $O(c)$ by $\sqrt{c}$ since $O(c)$ is `introduced' only later.
Also, in the proof of Lemma~\ref{cyclelemma} have $>(2\eta-\sqrt{\nu})k$ instead of
$(3/8-\sqrt{c})k$.}
$(3/8-\sqrt{c})k$. However there may be double edges, so delete edges to obtain an
oriented graph (which we still call $R^*_B$) in which each $x\in A''_0$ has indegree and outdegree at least
$(3/16-\sqrt{c})k \ge k/10$ (say, although we only need these degrees to be non-zero).
Now applying Lemma~\ref{find-cycle} to the oriented graphs $R$, $R^*_B$ (playing the role of $R^*$)
and $G_B^*$ (playing the role of $G^*$)
gives a directed Hamilton cycle $C_B$ in $G_B^*$.

Choose any vertex $v$ of $G_B^*$ and let $v^+$ be the vertex that succeeds
$v$ on $C_B$. Let $G^*_A$ be the digraph obtained from $G_A'''$ by
adding a new vertex $v_*$ which has the following neighbours:
the outneighbourhood of $v_*$ is $N^+_{G}(v) \cap A'$ and the inneighbourhood
is $N^-_{G}(v^+) \cap A'$.
Let $R_A^*$ be the digraph obtained from
$R_A$ by adding a new vertex $v_*$ and edges as follows.
Add an edge from $v_*$ to $i \in [\ell]$ if
$|N^+_{G}(v) \cap A_i'| > c|A_i'|$ and an edge from $i \in [\ell]$ to $v_*$
if $|N^-_{G}(v^+) \cap A_i'| > c|A_i'|$. Using Claim~1.1 it is easy to see that $v_*$ has indegree
and outdegree at least $(3/8-\sqrt{c})\ell$.
Again, we can delete edges to arrive
at an oriented graph $R_A^*$ in which $v_*$ has indegree and outdegree at least
$(3/16-\sqrt{c})\ell \ge \ell/10$. Now Lemma \ref{find-cycle} applied to $R_A$
(playing the role of $R$), $\mc{C}_A$ (playing the role of $\mc{C}$),
$R^*_A$ (playing the role of $R^*$), $G_A^*$ (playing the role of $G^*$) and with
$1/|G^*_A|,\ell,\eps_0,d_0,d^2/2$ playing the roles of $\eps_0,k,\eps,d,c$
gives a directed Hamilton cycle $C_A$ in $G_A^*$.
Note that this yields a Hamilton path in $G_A'''$ which starts in
$v_*^+$ and ends in $v_*^-$, where $v^{+}_*$ and $v^{-}_*$ are the
successor and predecessor of $v_*$
on $C_A$. Now $v^+ C_B vv^{+}_* C_A v^{-}_*v^+$ is a Hamilton cycle in $G^*$ and thus in $G$.
This contradiction completes the analysis of Case~1.

Note that the argument in the proof of~Case~1 implies the following lemma, which will be
used in~\cite{chvatal} to prove an approximate analogue of Chv\'atal's theorem on
hamiltonian degree sequences for digraphs.

\begin{lemma}\label{cyclelemma}
Let $M',n_0$ be positive numbers and let $\varepsilon,d,\eta,\nu,\tau$ be positive constants such that
$1/n_0\ll 1/M'\ll  \varepsilon \ll d\ll \nu\le \tau\ll \eta< 1$. Let~$G$ be an oriented graph on $n\ge n_0$
vertices such that $\delta^0(G)\ge 2\eta n$. Let~$R'$ be the reduced digraph of $G$ with
parameters $(\varepsilon,d)$ and such that $|R'|\ge M'$. Suppose that there exists a spanning
oriented subgraph~$R$ of~$R'$ with $\delta^0(R)\ge \eta |R|$ and such that $|N^+_R(S)|\ge |S|+\nu|S|$ for all
sets $S\subseteq V(R)$ with $\tau |R|< |S|< (1-\tau)|R|$.
Then $G$ contains a Hamilton cycle.
\end{lemma}

Let us now continue with the second case of the proof of Theorem~\ref{main}.

\medskip

\nib{Case 2.} \emph{There is a set $S \subset [k]$ with $k/3 < |S| < 2k/3$ and
$|N^+_R(S)| < |S| + 2ck$.}

\medskip

The strategy in this case is as follows: by exploiting the minimum semi-degree condition, we
will first show that $G$ has roughly a similar structure as the extremal example described in
Section~\ref{extremal}. In a sequence of further claims, we will then either find a Hamilton cycle
or obtain further structural information on $G$ which means that it must be
even more similar to the extremal example. Eventually, we arrive at a contradiction, since
being (almost) exactly like the extremal example is incompatible with the minimum semi-degree condition.
Unless stated otherwise, all neighbourhoods, degrees and numbers of edges refer to the oriented graph~$G$
from now on. Let
$$A_R := S \cap N^+_R(S), \ B_R := N^+_R(S) \sm S, \
C_R := [k] \sm (S \cup N^+_R(S)), \ D_R := S \sm N^+_R(S).
$$
Let $A := \bigcup_{i \in A_R} V_i$ and define $B,C,D$ similarly.
By definition we have $e_R(A_R,C_R)=e_R(A_R,D_R)=e_R(D_R,C_R)=e_R(D_R)=0$.
Since $R$ has parameters $(\eps^2/3,d)$, Lemma \ref{red-orient}(c,d) implies that we have
\begin{equation} \label{lowdens}
e(A,C), e(A,D), e(D,C), e(D) < 3dn^2.
\end{equation}
From now on we will
not calculate explicit constants multiplying $c$, and just write $O(c)$.
The constants implicit in the $O(*)$ notation will always be absolute.

\medskip

\nib{Claim 2.1.} $|A|, |B|, |C|, |D| = (1/4 \pm O(c))n$.

\begin{proof} First we show that $A_R$, $B_R$, $C_R$ and $D_R$ are non-empty.
Since the average value of $|N^+_R(x) \cap S|$ over all $x \in S$ is less than
$|S|/2$, we have
$$
|B_R| > \delta^+(R) - |S|/2 \stackrel{(\ref{eq:minR})}{>} (3/8 - 2d)k - k/3 > k/30.
$$
Also $|D_R| = |B_R| + |S| - |N^+_R(S)| > k/30 - 2ck > 0$.
Since $e_R(D_R)=0$, for any $x \in D_R$ we have
$|N_R(x)|\le  |A_R|+|B_R|+|C_R| = |N^+_R(S)| + |C_R|$.
Thus
$$
|C_R| \stackrel{(\ref{eq:minR})}{>}2(3/8-2d)k- |N^+_R(S)|> 2(3/8-2d)k-(2/3+2c)k > 0
$$
and also
$|A_R| = |C_R| + |N^+_R(S)| + |S| - k > 2(3/8-2d)k + |S| - k > 0$.

Pick a vertex $u_R \in D_R$ whose degree in~$R$ is minimal, a vertex
$v_R \in A_R$ whose outdegree in $R$ is minimal and a vertex $w_R \in C_R$ of whose indegree
in $R$ is minimal. Since the minima are at most the averages, inequality~(\ref{eq:minR})
implies that
$2(3/8-2d)k < d(u_R) \le |A_R|+|B_R|+|C_R|$,
$(3/8-2d)k < d^+(v_R) \le |A_R|/2+|B_R|$ and
$(3/8-2d)k < d^-(w_R) \le |B_R|+|C_R|/2$. We also have the inequality
$|B_R|-|D_R|=|N^+_R(S)|-|S| < 2ck$. Thus we may define positive reals
$r_A$, $r_B$, $r_C$, $r_D$ by
\begin{itemize}
\item $r_A := |A_R|/2+|B_R| - (3/8-2d)k$,
\item $r_B := \frac{3}{2}(|D_R|-|B_R|+2ck)$,
\item $r_C := |B_R|+|C_R|/2 - (3/8-2d)k$,
\item $r_D := |A_R|+|B_R|+|C_R| - 2(3/8-2d)k$.
\end{itemize}
Then
$$
r_A + r_B + r_C + r_D = \frac{3}{2}(|A_R|+|B_R|+|C_R|+|D_R|) + 3ck
- 4(3/8-2d)k = (3c+8d)k < 4ck.
$$
This in turn implies that
\begin{itemize}
\item $|D_R| = k - (|A_R|+|B_R|+|C_R|) = k - 2(3/8-2d)k - r_D
= k/4 \pm 5ck$,
\item $|B_R| = |D_R| + 2ck - \frac{2}{3}r_B = k/4 \pm 10ck$,
\item $|A_R| = 2((3/8-2d)k - |B_R| + r_A) = k/4 \pm 30ck$ and
\item $|C_R| = 2((3/8-2d)k - |B_R| + r_C) = k/4 \pm 30ck$.
\end{itemize}
Altogether, this gives
$|A|, |B|, |C|, |D| = (1/4 \pm 31c)n$.  \end{proof}

\medskip

\nib{Claim 2.2.} \begin{itemize}
\item $e(A)  > (1/2 - O(c))n^2/16$,
\item $e(A,B)  > (1 - O(c))n^2/16$,
\item $e(B,C) > (1 - O(c))n^2/16$,
\item $e(B,D)  > (1/2 - O(c))n^2/16$,
\item $e(C)  > (1/2 - O(c))n^2/16$,
\item $e(C,D)  > (1 - O(c))n^2/16$,
\item $e(D,A)  > (1 - O(c))n^2/16$,
\item $e(D,B)  > (1/2 - O(c))n^2/16$.
\end{itemize}

\begin{proof} Since $e(A,C), e(A,D) < 3dn^2$ by~(\ref{lowdens}) we have
$$
\sum_{x \in A} d^+(x) \le |A|^2/2 + |A||B| + 6dn^2 = (3/2+O(c))n^2/16.
$$
On the other hand, $\sum_{x \in A} d^+(x) \ge |A|(3n-4)/8 = (3/2-O(c))n^2/16$.
So we must have $e(A) > (1/2 - O(c))n^2/16$ and $e(A,B) > (1 - O(c))n^2/16$.
Also, since $e(A,C), e(D,C) < 3dn^2$ we have
$$(3/2-O(c))n^2/16 < \sum_{x \in C} d^-(x) < |B||C|+|C|^2/2+6dn^2
= (3/2+O(c))n^2/16,
$$ so $e(C) > (1/2 - O(c))n^2/16$ and
$e(B,C) > (1 - O(c))n^2/16$. Next, writing $\ov{D}:=A\cup B\cup C$ and using the inequalities
$e(D),e(D,C),e(A,D)<3dn^2$ gives
\begin{align*}
(3-O(c))n^2/16 & < e(D,\ov{D}) + e(\ov{D},D) +2e(D)\\
& < e(D,A)+ e(D,B) + e(B,D) + e(C,D) + 12dn^2 \\
& \le |D|(|A|+|B|+|C|) + 12dn^2
= (3+O(c))n^2/16,
\end{align*}
so $e(D,A) > (1 - O(c))n^2/16$ and $e(C,D) > (1 - O(c))n^2/16$.
Finally, since $e(D,C),e(D)<3dn^2$ we have
$$
(3/2-O(c))n^2/16 < \sum_{x \in D} d^+(x) < |A||D| + e(D,B)+ 6dn^2
<(1+O(c))n^2/16+e(D,B)
$$
and so $e(D,B) > (1/2 - O(c))n^2/16$.
Since $e(A,D),e(D)<3dn^2$ we have
$$
(3/2-O(c))n^2/16 < \sum_{x \in D} d^-(x) < e(B,D) + |C||D|+ 6dn^2
<(1+O(c))n^2/16+e(B,D)
$$
and so $e(B,D) > (1/2 - O(c))n^2/16$. \end{proof}

\medskip

Henceforth we will only use Claims~2.1 and~2.2 and make no further use of the information
in~\eqref{lowdens}. This has the advantage of making our picture invariant under
the relabelling $A \lra C$, $B \lra D$. Also, we add $V_0$ to $A$.
Since $|V_0| \le  (\eps^2/3) n$ and $\eps \ll c$, this does not affect the assertions of
Claims~2.1 and~2.2 (except as usual the constant in the $O(c)$-notation).
It will sometimes be convenient to use the notation
$A:=P(1)$, $B:=P(2)$, $C:=P(3)$, $D:=P(4)$.

Given a vertex $x\in G$, we will use the compact notation \emph{$x$ has property
$W \ns{:} A^*_* B^*_* C^*_* D^*_*$}
as follows. The notation starts with some $W \in \{A,B,C,D\}$, namely
the set that $x$ belongs to. Next,
for each of $A,B,C,D$ its superscript symbol refers to
the intersection with $N^+(x)$ and its subscript to the intersection
with $N^-(x)$. A symbol `$\ \ns{>} \alpha$' describes an intersection of
size at least $(\alpha-O(\sqrt{c}))n/4$ and a symbol `$\ \ns{<} \alpha$'
describes an intersection of size at most $(\alpha+O(\sqrt{c}))n/4$.
The absence of a symbol means that no statement is made
about that particular intersection, and we can omit any of $A,B,C,D$ if
it has no superscript or subscript. For example, to say
$x$ has property $B\ns{:}A^{>1}C^{<1/2}_{<1/3}$ means that
$x \in B$, $|N^+(x) \cap A| > (1-O(\sqrt{c}))n/4$,
$|N^+(x) \cap C| < (1/2+O(\sqrt{c}))n/4$ and $|N^-(x) \cap C| < (1/3+O(\sqrt{c}))n/4$.

We say that a vertex $x$ is \emph{cyclic} if it satisfies
$$P(i)\ns{:}P(i+1)^{>1}P(i-1)_{>1}
$$
for some $i\le 4$ (counting modulo 4). 
Claim 2.2 implies that at most $O(\sqrt{c})n$ vertices are
not cyclic. 
We say that a vertex is \emph{acceptable} if it has one of the following properties
\begin{itemize}
\item $A\ns{:}B^{>1/100}D_{>1/100}$, $A\ns{:}A^{>1/100}D_{>1/100}$,
$A\ns{:}A_{>1/100}B^{>1/100}$, $A\ns{:}A^{>1/100}_{>1/100}$,
\item $B\ns{:}A_{>1/100}C^{>1/100}$, $B\ns{:}A_{>1/100}D^{>1/100}$, $B\ns{:}C^{>1/100}D_{>1/100}$,
$B\ns{:}D^{>1/100}_{>1/100}$,
\item $C\ns{:}B_{>1/100}D^{>1/100}$,
$C\ns{:}B_{>1/100}C^{>1/100}$, $C\ns{:}C_{>1/100}D^{>1/100}$,
$C\ns{:}C^{>1/100}_{>1/100}$,
\item $D\ns{:}A^{>1/100}C_{>1/100}$, $D\ns{:}A^{>1/100}B_{>1/100}$, $D\ns{:}B^{>1/100}C_{>1/100}$,
$D\ns{:}B^{>1/100}_{>1/100}$.
\end{itemize}
In other words, a vertex is acceptable if it has a significant outneighbourhood in one of its
two {\em out-classes} and a significant inneighbourhood in one of its two {\em in-classes},
where `out-classes' and `in-classes' are to be understood with reference to the
extremal oriented graph in Section~\ref{extremal}. So for example $A$ has out-classes $A$
and $B$ and in-classes $A$ and $D$.
We will also call an edge {\em acceptable} if it is of the type allowed in the extremal
oriented graph (so for example an edge from $A$ to $B$ is acceptable but an edge from $B$ to $A$ is not).
Note that a cyclic vertex is also acceptable.

In what follows, we will carry out
$O(\sqrt{c})n$ reassignments of
vertices between the sets $A$, $B$, $C$ and $D$ (and a similar number of `path contractions' as
well). After each sequence of such reassignments
it is understood that the hidden constant in the $O(\sqrt{c})$-notation of the definition
of an acceptable/cyclic vertex is increased.

\medskip

\nib{Claim 2.3.} \emph{By reassigning vertices that are not cyclic to
$A$, $B$, $C$ or $D$ we can arrange that every vertex of~$G$ is acceptable. We can
also arrange that there is no vertex that is not cyclic but would
become so if it was reassigned.}

\begin{proof} We start by making any reassignments necessary to
satisfy the second statement of the claim. (To do this, we reassign all the vertices in
question in one step. As we are reassigning $O(\sqrt{c})n$ vertices we can then change
the hidden constant in the $O(\sqrt{c})$-notation to make sure that the vertices
which were cyclic before are still cyclic.)

To satisfy the first statement of the claim, for any vertex $x\in G$ we
let $P^+_x := \{1 \le i \le 4: |N^+(x) \cap P(i)| > n/400 \}$,
$P^-_x := \{1 \le i \le 4: |N^-(x) \cap P(i)| > n/400 \}$, and
$P_x := P^+_x \cup P^-_x$. By the minimum semi-degree condition
$|P^+_x| \ge 2$, $|P^-_x| \ge 2$ and $|P_x| \ge 3$. If there is
some $i$ such that $i+1 \in P^+_x$ and $i-1 \in P^-_x$ (where we use addition and
subtraction mod $4$) then we can put $x$ into $P(i)$ and it will have
property $P(i)\ns{:}P(i-1)_{>1/100}P(i+1)^{>1/100}$, i.e.~$x$ will become
acceptable. Otherwise we must
have either $P^+_x = \{1,3\}$ and $P^-_x = \{2,4\}$, or
$P^-_x = \{1,3\}$ and $P^+_x = \{2,4\}$. In either case we can put $x$
into $A=P(1)$: in the first case it will have property $A\ns{:}A^{>1/100}D_{>1/100}$
and in the second case property $A\ns{:}A_{>1/100}B^{>1/100}$. So in both cases~$x$
will become acceptable. As before, by increasing
the hidden constant in the $O(\sqrt{c})$-notation if necessary, we can ensure that
the properties of all other vertices are maintained. \end{proof}

\medskip

In the proof of the following claims, contractions of paths will play an
important role. Given a path $P$ in $G$ whose initial vertex $p_1$ and final vertex
$p_2$ lie in the same class $P(i)$, the \emph{contraction of $P$} yields the following oriented
graph~$H$: we add a new vertex $p$ to the class $P(i)$ and remove (the vertices of) the path $P$ from $G$.
$N^+_H(p)$ will be $N^+(p_2)\cap P(i+1)$ and $N^-_H(p)$ will be $N^-(p_1)\cap P(i-1)$.
Note that any cycle in $H$ which includes $p$ corresponds to a cycle in $G$ which
includes $P$. The paths $P$ that we use in the proof of Claim 2.4 will have
initial and final vertices in the same class and will be {\em acceptable},
meaning that every edge on~$P$ is acceptable. Note that each such acceptable path~$P$ must be
\emph{$BD$-balanced}, meaning that if we delete the initial vertex of $P$
we are left with a path that meets $B$ and $D$ the same number of times.
This may be seen from the observations that visits of $P$ to $B \cup D$ alternate
between $B$ and $D$, and if the path is in $A$ and then leaves it must visit
$B$ and $D$ an equal number of times before returning to $A$ (and similarly for $C$).
Note that if we have $|B|=|D|$ and contract a $BD$-balanced path, then the resulting digraph
will still satisfy $|B|=|D|$. The `moreover' part of the following claim is used later in the proof
to turn a graph with $|B|=|D|+1$ into one with $|B|=|D|$ under certain circumstances.

\medskip

\nib{Claim 2.4.} \emph{If $|B|=|D|$ and every vertex is acceptable
then $G$ has a Hamilton cycle.
Moreover, the assertion also holds under the following slightly weaker assumption:
$|B|=|D|$ and there exists some vertex $x$ such that every vertex except~$x$ is acceptable,
there is at least one acceptable edge going into $x$ and
at least one acceptable edge coming out of $x$.}

\begin{proof} We will use the `standard' version of the Blow-up Lemma (Lemma~\ref{standardblowup})
to prove the first assertion.
For this, the idea is to first find suitable paths which together contain all the non-cyclic vertices.
We will contract these paths into vertices so that the resulting oriented graph $G_1$ consists
entirely of cyclic vertices.
Then we will find suitable paths whose contraction results in an oriented graph $G_2$ which
satisfies $|A|=|B|=|C|=|D|$ and all of whose vertices are cyclic.
We can then apply the Blow-up Lemma to its underlying graph to find a directed Hamilton cycle in~$G_2$
which `winds around' $A,B,C,D$.
The choice of our paths will then imply that this Hamilton cycle corresponds to one in~$G$.

Let $v_1,\dots,v_t$ be the vertices which are acceptable but not cyclic.
For each $v_i$, choose a cyclic outneighbour $v_i^+$
and a cyclic inneighbour $v_i^-$ so that all of these vertices are distinct
and so that the edges $v_iv^+_i$ and $v^-_iv_i$ are acceptable.
Note that this can be done since $t=O(\sqrt{c})n$.
Let $P_i'$ be a path of length at most 3 starting at $v_i^+$ and ending at a
cyclic vertex which lies in the same class as $v_i^-$, and where the successive vertices lie in
successive classes, i.e.~the successor of a vertex $x \in V(P) \cap P(i)$ lies in $P(i+1)$.
(So if for example $v_i$ has the first of the
acceptable properties of a vertex in~$A$, then we can choose $v^+_i\in B$, $v^-_i\in D$ and
so $P'_i$ would have its final vertex in $D$. Also, if $v_i^-$ and $v_i^+$ lie in the same class
then $P_i'$ consists of the single vertex $v_i^+$.)
Again, the paths $P_i'$ can be chosen to be disjoint. Let $P_i:=v_i^-v_iv_i^+P_i'$.
Then the $P_i$ are acceptable and thus $BD$-balanced.
Let $G_1$ be the oriented graph obtained from $G$ by contracting the paths $P_i$.
Then every vertex of $G_1$ is cyclic
(by changing the constant involved in the $O(\sqrt{c})$ notation in the definition of
a cyclic vertex if necessary). Moreover, the sets $A,B,C,D$ in $G_1$ still satisfy $|B|=|D|$
and we still have that the sizes of the other pairs of sets differ by at most $O(\sqrt{c})n$.

Now suppose that $|A| < |C|$ and let $s:=|C|-|A|$. Greedily find a path $P_C$ of
the form
$$\underbrace{CCDAB \dots CCDAB}_{s \text{ times}}C$$
consisting entirely of cyclic vertices. So~$P_C$ starts with an edge between two cyclic vertices in~$C$.
(Claim~2.2 implies that almost all (unordered) pairs of (cyclic) vertices in~$C$ are
joined by an edge.) Then the path~$P_C$ uses one cyclic vertex in~$D$, one in~$A$ etc.  
Let $G_2$ be the digraph obtained by contracting $P_C$.
Then in $G_2$, we have $|A|=|C|$ and $|B|=|D|$. If $|A|>|C|$ we can achieve equality in the similar way
by contracting a path~$P_A$ of the form~$AABCD\dots AABCDA$. Note that since $s=O(\sqrt{c})n$,
all vertices of $G_2$ are still cyclic.
Now suppose that in $G_2$ we have $|B|>|A|$. Let $s:=|B|-|A|$.
This time, we greedily find a path $P_B$ of the form
$$\underbrace{BDABCDBCDA \dots BDABCDBCDA}_{s \text{ times}}B$$
consisting of cyclic vertices. To see that such a path exists, note that Claim~2.2 implies
that there are at least~$|B|/4$ (say) cyclic vertices in~$B$ with at least~$|D|/4$ cyclic
outneighbours in~$D$. So~$P_B$ starts in such a cyclic vertex in~$B$ and then uses cyclic
vertices in~$D$, $A$, $B$ and~$C$. Since there are at least~$|D|/4$
cyclic vertices in~$D$ having at least~$|B|/4$ cyclic outneighbours in~$B$, $P_C$ can then
move to such a cyclic vertex in~$D$ and use cyclic vertices in~$B$, $C$, $D$ and~$A$ etc. 
Note that $P_B$ is $BD$-balanced.
By contracting $P_B$, we obtain an oriented graph (which we still call $G_2$) with
$|A|=|B|=|C|=|D|$ and all of whose vertices are still cyclic.
Finally, suppose that in $G_2$ we have $|B|<|A|$. In this case we can equalize the sets
by contracting two paths~$P_A$ and~$P_C$ as above.

Thus we have arranged that $|A|=|B|=|C|=|D|$ in~$G_2$.
Let $G_2'$ be the underlying graph corresponding to  the set of edges oriented from $P(i)$ to $P(i+1)$, for
$1 \le i \le 4$.
Since all vertices of $G_2$ are cyclic and we chose $c \ll \eta \ll 1$, each
pair $(P(i),P(i+1))$ is $(\eta,1)$-super-regular in $G'_2$.
Also, $G_2'$ is simple, i.e.~there are no multiple edges.
Let $F'$ be the $4$-partite graph with vertex classes
$A=P(1),B=P(2),C=P(3),D=P(4)$ where the 4 bipartite graphs induced by $(P(i),P(i+1))$ are all complete.
Clearly $F'$ has a Hamilton cycle, so we can apply Lemma~\ref{standardblowup}
with $k=4$, $\Delta=2$ to find a Hamilton cycle $C_{\rm Ham}$ in~$G_2'$. The `moreover' part of
Lemma~\ref{standardblowup} implies that we can assume that~$C_{\rm Ham}$
continually `winds around' $A,B,C,D$, i.e.~one neighbour on $C_{\rm Ham}$
of a vertex $x \in P(i)$ lies in $P(i+1)$
and the other in $P(i-1)$. $C_{\rm Ham}$ corresponds to a
directed Hamilton cycle in $G_2$ and thus in turn to a directed Hamilton cycle in $G$.

Now, we deduce the `moreover' part from the first part of Claim~2.4.
Similarly as in the first part, the approach is to find a suitable path~$P$
containing~$x$ which we can contract into a single vertex so that the resulting oriented graph
still satisfies $|B|=|D|$ and now all of its vertices are acceptable.
Choose $x^-$ and $x^+$ so that $x^-x$ and $xx^+$ are acceptable edges.
Since $x^-$ is acceptable, it has a cyclic inneighbour $x^{--}$ so that the edge
$x^{--}x^-$ is acceptable. Let $P(i)$ be the class which contains $x^{--}$.
Let $P'$ be a path of length at most~3 starting at $x^+$ and ending at an
cyclic vertex in~$P(i)$
so that successive vertices lie in successive classes.
Let $P:=x^{--}x^-xx^+P'$. Then $P$ is acceptable and thus $BD$-balanced.
Let $H$ be the oriented graph obtained from $G$ by contracting $P$.
Then in $H$ we still have $|B|=|D|$. All vertices that were previously acceptable/cyclic
are still so (possibly with a larger error term $O(\sqrt{c})$).
Since $x^{--}$ and the terminal vertex of $P$ are both cyclic (and thus acceptable),
this means that the new vertex resulting
from the contraction of $P$ is still acceptable.
So we can apply the first part of the claim to obtain a Hamilton cycle in $H$ which clearly
corresponds to one in the original oriented graph $G$.
\end{proof}

The picture is still invariant under the relabelling
$A \lra C$, $B \lra D$, so we may assume that $|B| \ge |D|$.
Since we are assuming there is no Hamilton cycle, Claim~2.4 gives
$|B| > |D|$.

\medskip

\nib{Claim 2.5.} \emph{For each of the properties
$A\ns{:}C_{>1/100}$, $A\ns{:}B_{>1/100}$,
$C\ns{:}A^{>1/100}$, $C\ns{:}B^{>1/100}$,
there are less than $|B|-|D|$ vertices with that property.}

\begin{proof} Suppose to the contrary that for example we can find $|B|-|D|=:t$ vertices
$v_1, \dots, v_t$ in $A$ having the property $A\ns{:}C_{>1/100}$.
Select distinct cyclic vertices
$v_1^-, \dots, v_t^-$ in $C$ with $v_i^- \in N^-(v_i)$ for $1 \le i \le t$.
This can be achieved greedily, as $|B|-|D|=O(\sqrt{c})n < n/1000$, say.
Since $v_1, \dots, v_t$ are acceptable vertices in $A$ we can greedily select
distinct cyclic vertices $v_i^+ \in N^+(v_i) \cap (A \cup B)$.
For each $i$ such that $v_i^+ \in B$ let
$b_i^+:= v_i^+$. For each $i$ such that $v_i^+\in A$
select $b_i^+$ to be a cyclic vertex in $N^+(v_i^+) \cap B$;
we can ensure that $b_1^+, \dots, b_t^+$ are distinct. Similarly
we can select cyclic vertices $b_i^{-} \in N^-(v_i^-) \cap B$ which are
distinct from each other and from $b_1^+, \dots, b_t^+$. Thus we have constructed
$t=|B|-|D|$ vertex disjoint paths $P_1, \dots, P_t$ where $P_i$ starts at
$b_i^-$ in $B$, goes through~$C$, then it uses~$1$ or $2$ vertices of $A$, and it ends
at $b_i^+$ in $B$.
Consider a new oriented graph $H$, obtained from $G$ by contracting these paths.
Note that the vertices of~$H$ correponding to the paths~$P_i$ are acceptable and the analogues
of~$B$ and~$D$ in~$H$ now have equal size. So we can apply Claim~2.4 to find a Hamilton cycle in~$H$.
This corresponds to a Hamilton cycle in $G$, which
contradicts our assumption. Therefore there are less than $|B|-|D|$ vertices
with property $A\ns{:}C_{>1/100}$.

The statement for property $C\ns{:}A^{>1/100}$ follows in a similar
way, as we can again start by finding a matching of size $|B|-|D|$
consisting of edges directed from $C$ to $A$.
The arguments for the other two properties are also similar.
For instance, if we have $|B|-|D|$ vertices $v_1,\dots,v_t$ in
$A$ having property $A\ns{:}B_{>1/100}$ we can find
a matching $b_1^- v_1, \dots ,b_t^-v_t$ of edges directed from
$B$ to $A$ such that the $b^-_i$ are cyclic. We can extend this to $|B|-|D|$ vertex disjoint
paths $P_1, \dots, P_t$, where $P_i$ starts with the edge $b_i^- v_i$,
it then goes either directly to a cyclic vertex in $B$ or it uses one more
vertex from~$A$ before it ends at a cyclic vertex in $B$.%
\COMMENT{previous version had also a backwards extension from $v_i^-$ but doesnt seem to be necessary}
Now we construct a new oriented graph~$H$ similarly as before
and find a Hamilton cycle in $H$ and
then in~$G$. \end{proof}

We now say that a vertex is \emph{good} if it is acceptable, and also has one of the
properties
$$
A\ns{:}B_{<1/100}C_{<1/100}, \ \
B\ns{:}A^{<1/100}B^{<1/100}_{<1/100}C_{<1/100}, \ \
C\ns{:}A^{<1/100}B^{<1/100} \ \mbox{ or } \
D\ns{:}$$
(The last option means that every acceptable vertex in $D$ is automatically good.)
Note that a cyclic vertex is not necessarily good.

\medskip

\nib{Claim 2.6.} \emph{By reassigning at most $O(\sqrt{c})n$ vertices we can arrange that
every vertex is good.}


\begin{proof} While $|B|>|D|$ we reassign vertices that are bad (i.e.~not good)
as follows. Suppose there is a bad vertex $x \in B \cup C$ with
$|N^+(x) \cap A| > n/400$. If we also have
$|N^-(x) \cap C| > n/400$ or $|N^-(x) \cap B| > n/400$,
then by reassigning $x$ to~$D$ it will become a good vertex.
(Recall that for a vertex in $D$, being acceptable and good is the same.)
If not, then  $|N^-(x) \cap C|, |N^-(x) \cap B| \le n/400$,
and so $|N^-(x) \cap A| \ge n/400$.
In this case we can make~$x$ good by reassigning it to~$A$.
Exactly the same argument works%
    \COMMENT{Its really the same argument: we put $x$ into~$D$ if
$|N^-(x)\cap C|\ge n/400$ or $|N^-(x)\cap B|\ge n/400$ and put it into~$A$ otherwise.}
if there is a bad vertex $x \in B \cup C$
with $|N^+(x) \cap B| > n/400$.
If $x$ was in $B$ we have decreased $|B|-|D|$ by $1$ or $2$, whereas by Claim~2.5
we will make at most $O(\sqrt{c})n$ such reassignments of vertices
$x\in C$.

Similarly, suppose there is a bad vertex $x \in A \cup B$ with
$|N^-(x) \cap B| > n/400$ or $|N^-(x) \cap C| > n/400$.
If we also have
$|N^+(x) \cap A| > n/400$ or $|N^+(x) \cap B| > n/400$, then by
reassigning~$x$ to~$D$ it becomes good. If not, then
$|N^+(x) \cap A|, |N^+(x) \cap B| \le n/400$,
and so $|N^+(x) \cap C|, |N^+(x) \cap D| \ge n/400$.
In this case we can reassign $x$ to~$C$. Again, if $x$ was
in $B$ we have decreased $|B|-|D|$ by $1$ or $2$, whereas by Claim~2.5
we will make at most $O(\sqrt{c})n$ such reassignments of vertices
$x\in A$.

The above cases cover all possibilities of a vertex being bad.
If during this process $|B|$ and $|D|$ become equal we can
apply Claim~2.4 to find a Hamilton cycle. Alternatively,
it may happen that $|B|-|D|$ goes from $+1$ to $-1$ if a vertex $x$
is moved from $B$ to $D$. In this case we claim that we can put
$x$ into $A$ or $C$ to achieve the following property: every
vertex except $x$ is acceptable,
there is at least one acceptable edge going into $x$ and
at least one acceptable edge coming out of $x$.

To see this, suppose first that one of $N^-(x) \cap A$ or $N^-(x) \cap D$
is non-empty. If one of $N^+(x) \cap A$ or $N^+(x) \cap B$ is
also non-empty we can put $x$ into~$A$.
Otherwise $N^+(x) \cap C$ and $N^+(x) \cap D$ are both non-empty. Now we
can put $x$ into~$C$ unless
$N^-(x) \cap B$ and $N^-(x) \cap C$ are both empty, which cannot
happen, as then $x$ would not have been a bad vertex of $B$. Now suppose that
$N^-(x) \cap A$ and $N^-(x) \cap D$ are both empty, so
$N^-(x) \cap B$ and $N^-(x) \cap C$ are both non-empty.
We can put $x$ into~$C$ unless
$N^+(x) \cap C$ and $N^+(x) \cap D$ are both empty. But this implies that~$x$
satisfies the property $B\ns{:}A^{>1}B^{>1/2}_{>1/2}C_{>1}$. This cannot
happen, as then $x$ would have been reassigned to $D$ as
a cyclic vertex in Claim~2.3. Therefore we can
apply the `moreover' part of Claim~2.4 to find a Hamilton cycle. But since we are
assuming there is no Hamilton cycle the process must terminate
with an assignment of vertices where all vertices are good. \end{proof}

\medskip

Let $M$ be a maximum matching consisting of edges in $E(B,A) \cup E(B)
\cup E(C,A) \cup E(C,B)$. Say that $M \cap E(B,A)$ matches $B_A \subset B$ with
$A_B \subset A$, that $M \cap E(B)$ is a matching on
$B_B \subset B$, that $M \cap E(C,A)$ matches $C_A \subset C$ with $A_C \subset A$
and that $M \cap E(C,B)$ matches $C_B \subset C$ with $B_C \subset B$.
Note that $e(M)=|A_B| + |A_C| + |B_B|/2 + |B_C|$.

We must have $e(M) < |B| - |D|$. Otherwise, by a similar
argument to that in Claim 2.5 we could extend
$t:=|B|-|D|$ edges $v_i^-v_i$ of $M$ to $|B|-|D|$ vertex disjoint
paths $P_1, \dots, P_t$, where $P_i$ includes $v_i^- v_i$,
starts and ends at cyclic vertices $b_i^-$, $b_i^+$ in $B$,
and uses two more vertices from $B$ than from $D$. Then
we would find a Hamilton cycle as in Claim 2.5, which would
be a contradiction.

\medskip

\nib{Claim 2.7.} \emph{$e(M)=0$ and $|B|-|D|=1$.}

\begin{proof} We will first prove that $e(M)=0$.
Assume to the contrary that $e(M) \ge 1$. Let $A' := A \sm (A_B \cup A_C)$,
$B' := B \sm (B_A \cup B_B \cup B_C)$ and $C' := C \sm (C_A \cup C_B)$.
By the maximality of $M$ there are no edges from $B' \cup C'$ to $A'$
or from $B' \cup C'$ to $B'$. Since all vertices are good this implies%
   \COMMENT{Potentially $e(M)=|C_A|=|A_C|$.}
\begin{align} \label{eAC}
e(C,A) & \le e(C_A \cup C_B,A)+e(C,A_B \cup A_C)
\le (|C_A \cup C_B|+|A_B \cup A_C|)(1/100+O(\sqrt{c})n/4\nonumber\\
& \le 2e(M)(n/4)/99\le e(M)|A|/49.
\end{align}
Similarly%
     \COMMENT{$e(B,A)\le e(B_A\cup B_B \cup B_C,A)+e(B,A_B\cup A_C)\le
(|B_A\cup B_B \cup B_C|+|A_B\cup A_C|)(n/4)/99\le 2e(M)(n/4)/99$}
$e(B,A) < e(M)|A|/49$. So by considering
the average indegree of the vertices in $A$ we can find a vertex $a\in A$
such that
$$d^-(a) < |A|/2 + |D| + 2e(M)/49.$$
Also, we can obtain $e(C,B) < e(M)|C|/49$ in a similar way as~(\ref{eAC}).
So by averaging, we can find a vertex $c \in C$ with
$$d^+(c) < |C|/2 +  |D| + 2e(M)/49.$$
By the maximality of~$M$, every edge in $B$ is incident with a vertex of $B_A \cup B_B \cup B_C$.
Together with the fact that every vertex is good, this implies that%
    \COMMENT{Could have $|B_B|=2e(M)$ and about $|B|/100$ in- and out- edges on each vertex}
$$e(B) < |B_A \cup B_B \cup B_C| \cdot 2(1/100+O(\sqrt{c}))n/4 < 2e(M)|B|/49.$$
So we can find $b \in B$ with
$$d(b) < |A|+|C|+|D|+4e(M)/49.$$
Now we use the minimum semi-degree condition and the above inequalities
to get
\begin{equation} \label{dadcdb}
\frac{3n-4}{2} \le d^-(a)+d^+(c)+d(b) < \frac{3}{2}(|A|+|C|+2|D|) + 8e(M)/49
\end{equation}
Substituting $n=|A|+|B|+|C|+|D|$ and simplifying%
     \COMMENT{$(3n-4)/2\le \frac{3}{2}(n+|D|-|B|)+8e(M)/49$}
gives
\begin{equation} \label{B-D}
|B|-|D| \le (2/3)(2+8e(M)/49)=4/3+16e(M)/147.
\end{equation}
On the other hand, we previously observed that $|B|-|D| \ge e(M)+1$.
Combining this gives $131e(M)/147 \le 1/3$, which is only possible if $e(M)=0$.
Now~(\ref{B-D}) immediately implies the claim.
\end{proof}

Note that $e(M)=0$ implies that $e(B \cup C,A)=0$.
So by averaging, there is a vertex $a \in A$ with
$d^-(a) \le (|A|-1)/2+|D|$. Since $e(M)=0$ also implies that $e(C,B)=0$ and $e(B)=0$
there is vertex $c \in C$ with
$d^+(c) \le (|C|-1)/2+|D|$ and a vertex $b \in B$
with $d(b) \le |A|+|C|+|D|$.
So as in~(\ref{dadcdb}), using $|D|=|B|-1$ we obtain that
$$
\frac{3n-4}{2} \le d^-(a)+d^+(c)+d(b) \le \frac{3}{2}(|A|+|C|+2|D|)-1=
\frac{3}{2}(n-1)-1,
$$
which is impossible. This contradiction to our initial assumption
on $G$ completes the proof. \qed

As remarked in the introduction, the proof can be modified to find a cycle of any given length~$\ell$
through any given vertex~$v$ of~$G$, 
where $\ell \ge n/10^{10}$, say. We give an outline of the necessary modifications here, the
details can be found in~\cite{lukethesis}.
In Case~1, we find the cycle as follows: consider a random subset $Q$ of $\ell-1$ vertices of~$G-v$
obtained by choosing $(\ell-1)/|R|$ vertices from each cluster $V_i$. 
With high probability, the oriented subgraph $G_Q$ induced by $Q \cup \{ v \}$ also has $R$ 
as a reduced oriented graph (with slightly worse parameters). 
Thus Lemma~\ref{cyclelemma} implies that $G_Q$ contains a Hamilton cycle.
So suppose we are in Case~2. 
The only significant difference in the argument is in the application of the Blow-up Lemma in 
the proof of Claim~2.4. After contracting paths so that all vertices are now acceptable
and $|A|=|B|=|C|=|D|$, we then 
find a single path $P_1$ which contains all the vertices which correspond to 
contracted paths and also contains~$v$. 
We then apply the Blow-up Lemma to find a path $P_2$ which (i) contains all but $n-\ell$
of the remaining vertices, (ii) whose initial vertex lies in the outneighbourhood of the final 
vertex of $P_1$ and (iii) whose final vertex lies in the inneighbourhood of the initial vertex
of~$P_1$. Then $P_1\cup P_2$ corresponds to a cycle of length~$\ell$ through~$v$ in~$G$.

\medskip

{\footnotesize \obeylines \parindent=0pt

\begin{tabular}{lll}

Peter Keevash                       &\ &  Daniela K\"{u}hn \& Deryk Osthus \\
School of Mathematical Sciences     &\ &  School of Mathematics \\
Queen Mary, University of London    &\ &  University of Birmingham \\
Mile End Road                       &\ &  Edgbaston \\
London                              &\ &  Birmingham \\
E1 4NS                              &\ &  B15 2TT \\
UK                                  &\ &  UK \\

\end{tabular}
}

{\footnotesize \parindent=0pt

\it{E-mail addresses}:
\tt{p.keevash@qmul.ac.uk}, \tt{\{kuehn,osthus\}@maths.bham.ac.uk}}

\end{document}